\documentclass[fleqn,leqno]{elsarticle}
  \biboptions{sort&compress}
 \pdfoutput=1

\usepackage[utf8]{inputenc}
\usepackage[USenglish]{babel}

\usepackage[centertags]{amsmath}
\usepackage{amsthm, amssymb}
\usepackage{dsfont}
\usepackage{bbm}
\usepackage{comment}
\usepackage{color}
\usepackage[normalem]{ulem}

\usepackage[pdftex,
            pdfauthor={T.Rippl, A.Munk, A.Sturm},
            pdftitle={Limit Laws of the empirical distance: Gaussian distributions},
            pdfproducer={pdfLaTex},
            pdfcreator={}]{hyperref}

\theoremstyle{plain}
\newtheorem{theorem}{Theorem}[section]
\newtheorem{lemma}[theorem]{Lemma}
\newtheorem{proposition}[theorem]{Proposition}
\newtheorem{corollary}[theorem]{Corollary}

\newtheorem*{conjecture*}{Conjecture}

\theoremstyle{definition}

\theoremstyle{remark}
\newtheorem{remark}{Remark}[section]

\newcommand{\IP}{\mathbb P}
\renewcommand{\P}{\mathbb P}
\newcommand{\IQ}{\mathbb Q}
\newcommand{\QQ}{\mathbb Q}
\newcommand{\Q}{\mathbb{Q}}
\newcommand{\IE}{\mathbb E}
\newcommand{\IR}{\mathbb R}
\newcommand{\IC}{\mathbb C}

\newcommand{\IN}{\mathbb N}

\newcommand{\N}{\mathbb N}
\newcommand{\GG}{\mathbb G}
\newcommand{\DD}{\mathbb D}
\newcommand{\EE}{\mathbb E}

\newcommand{\mcH}{\mathcal{H}}
\newcommand{\mcL}{\mathcal{L}}
\newcommand{\R}{\mathbb R}
\newcommand{\Frechet}{Fr\'echet{}}
\newcommand{\dimension}{d}
\newcommand{\Rdim}{\IR^{\dimension}}
\newcommand{\dx}{\mathrm{d}}
\newcommand{\mcM}{\mathcal{M}}
\DeclareMathOperator{\tr}{tr} 

\newcommand{\wass}{\mathcal{W}}
\newcommand{\Gwhat}{\widehat{\mathcal{GW}}}
\newcommand{\Gw}{\mathcal{GW}}

\newcommand{\1}{\mathds{1}}

\newcommand{\nn}{\nonumber}

\begin{document}

 \title{Limit laws of the empirical Wasserstein distance: Gaussian distributions}
 \author{Thomas Rippl\footnote{Institute for Mathematical Stochastics, Georg-August-Universit\"at G\"ottingen, G\"ottingen, Germany; trippl@uni-goettingen.de},
 Axel Munk\footnote{Institute for Mathematical Stochastics, Georg-August-Universit\"at G\"ottingen and Max Planck Institute for Biophysical Chemistry, G\"ottingen, Germany; munk@math.uni-goettingen.de},
 Anja Sturm\footnote{Institute for Mathematical Stochastics, Georg-August-Universit\"at G\"ottingen, G\"ottingen, Germany; asturm@math.uni-goettingen.de}}

\thispagestyle{empty}
\begin{abstract}
We derive central limit theorems for the Wasserstein distance between the empirical distributions of Gaussian samples.
The cases are distinguished whether the underlying laws are the same or different. Results are based on the (quadratic) Fr\'echet differentiability of the Wasserstein distance in the gaussian case. 
Extensions to elliptically symmetric distributions are discussed as well as several applications such as bootstrap and statistical testing.
\end{abstract}
 
 \maketitle
 
\\
\smallskip


\noindent {\bf Acknowledgement:} Support of DFG RTN2088 is gratefully acknowledged.
 We are grateful to a referee, the associate editor and Michael Habeck for helpful comments.

 \section{Introduction}
Let $\P,\Q$ be in  $\mcM_1(\R^d)$, the probability measures on $\R^d$.
Consider $\pi_i:\IR^d\times \IR^d \to \IR, x=(x_1,x_2) \mapsto x_i$, $i=1,2$, the projections on the first or the second $d$-dimensional vector,  and define
  \[ \Pi(\IP,\IQ) = \{ \mu \in \mathcal{M}_1(\IR^d\times \IR^d): \, \mu \circ \pi_1^{-1} = \IP, \mu \circ \pi_2^{-1} = \IQ\}\]
  as the set of probability measures on $\IR^d\times \IR^d$ with marginals $\IP$ and $\IQ.$
Then for $p \geq 1$ we define the \emph{$p$-Wasserstein distance} as
   \begin{eqnarray}
   \label{p-Wasserstein}
    \wass_p (\IP,\IQ) & := & \inf_{\mu \in \Pi(\IP,\IQ)} \left( \int_{\IR^{2d}} ||x-y||^p\, \mu(dx,dy)\right)^{1/p}  \, .
   \end{eqnarray}
There is a variety of interpretations and equivalent definitions of $\wass_p,$  for example as a mass transport problem; we refer the reader for extensive overviews to Villani~\cite{cV:07} and Rachev and R\"uschendorf~\cite{rachev1998mass}.

\smallskip

In this paper we are concerned with the \emph{statistical} task of estimating  $\wass_p (\IP,\IQ)$ from given data $X_1, \dotsc, X_n \sim \IP$ i.i.d.~(and possibly also from data $Y_1, \dotsc, Y_m \sim \Q$ i.i.d.)  and with the investigation of  certain characteristics of this estimate which are relevant for inferential purposes.
Replacing $\P$ by the empirical measure $\P_n$ associated with $X_1,\dotsc, X_n$ yields the empirical Wasserstein distance $\widehat{\wass}_{p,n}:=\widehat{\wass_p}(\IP_n, \IQ)$ which provides a natural estimate of $\wass_p(\P,\Q)$ for a given $\Q$.
Similarly, define  $\widehat{\wass}_{p,n,m} := \widehat{\wass}_p(\P_n,\Q_m)$ in the two sample case.
For inferential purposes (e.g.~testing or confidence intervals for $\wass_p(\P,\Q)$) it is of particular relevance to investigate the (asymptotic) distribution of the empirical Wasserstein distance.

\
This is meanwhile well understood for measures $\IP,\IQ$ on the real line $\R$  as in this case an explicit representation of the Wasserstein distance (and its empirical counterpart) exists (see e.g.~\cite{Frechet1951, Hoeffding1940, KR58, Vallender, Major1978487, Mallows}) 
\begin{eqnarray}\label{explonedimension}
\wass_p^p(\P,\Q)  =  \int_{[0,1]} |F^{-1}(t) - G^{-1}(t) |^p \, dt. 
\end{eqnarray}
Here,  $F(x) = \P((-\infty,x])$ and $G(x) = \Q((-\infty,x])$ for $x \in \R$ denote the c.d.f.s of $\P$ and $\Q$, respectively, and $F^{-1}$ and $G^{-1}$ its inverse quantile functions.
Now, $\widehat{\wass}_{p,n}$ is defined as in  \eqref{explonedimension} with $F^{-1}$ replaced by the empirical quantile function $F_n^{-1},$ and the representation \eqref{explonedimension} can be used to derive limit theorems based on the underlying quantile process $\sqrt{n}(F_n^{-1} -F^{-1})$. 
These results require a \emph{scaling rate} $(a_n)_{n\in \N}$ such that the laws 
\begin{equation}\label{scalinglimit}
a_n \left( \widehat{\wass_p^p}(\P_n,\Q) - \wass_p^p(\P,\Q)  + b_n \right), \text{ as } n\to \infty \end{equation} 
(for some centering sequence $(b_n)_{n \in \N}$) converge weakly to a (non-degenerate) \emph{limit distribution}.
Depending on whether $F=G$  as well as on the tail behavior of the distributions $F$ and $G$ we find ourselves in different asymptotic regimes. 
Roughly speaking, when $F=G$ (i.e.~$\P = \Q$, $\wass_p(\P,\Q)=0$), $a_n = n$ is the proper scaling rate, i.e.~the limit is of \emph{second order} and given 
by a weighted sum of $\chi^2$ laws (see e.g. \cite{del_barrio_central_1999, del2005asymptotics}).
In general, $b_n$ depends on the tail behavior of $F$. 
In contrast, when $F\neq G$, i.e.~$\wass_p^p(\P,\Q)>0$ for $a_n = \sqrt{n}$, $b_n =0$
the limit is of \emph{first order} and 
$\sqrt{n} (\widehat{\wass_p}(\IP_n, \IQ) - {\wass_p}(\IP, \IQ))$ is asymptotically normal
(see \cite{munk_nonparametric_1998, freitag2007nonparametric}) under appropriate tail conditions.
Various applications  of these and related distributional results, e.g.~for trimmed versions of the Wasserstein distance,  include the comparison of distributions and 
goodness of fit testing  (\cite{munk_nonparametric_1998, ABCM08, del2000contributions, freitag2005hadamard}), 
template registration (Section 4 in \cite{boissard2011distribution, Loubes}),  bioequivalence testing (\cite{freitag2007nonparametric}),   atmospheric research (\cite{zhou2011statistical}), or large scale microscopy imaging (\cite{ruttenberg2013quantifying}). 

\smallskip

In contrast to the real line ($d=1$), up to now limiting results as in  (\ref{scalinglimit}) remain elusive for $\R^d$, $d\geq 2$.
However, see \cite{AKT:84} and \cite{dobric_asymptotics_1995} for almost sure limit results and \cite{fournier_rate_2013} for moment bounds on $\widehat{\wass}_{p,n}$. Already the planar case $d=2$ is remarkably challenging (\cite{AKT:84}). One  difficulty
is that no simple characterization as in \eqref{explonedimension} via the (empirical) c.d.f's exists anymore. In particular, the couplings for which the infimum in (\ref{p-Wasserstein}) is attained are much more involved, see e.g.~\cite{knott1984optimal, ruschendorf1990characterization}.
We will come back to this in the context of our subsequent results later on.

In this article we aim to shed some light on the case $d\geq 2$ by further restricting the possible measures $\IP, \IQ$ to the Gaussians (and more generally to elliptical distributions). Here, a well known explicit representation of ${\wass_p}(\IP, \IQ)$ can be used (see e.g. \cite{dowson1982frechet}, \cite{olkin1982distance}, \cite{gelbrich1990formula}) which allows one to obtain explicit limit theorems again.  The Gaussian case is of particular interest as it provides, as shown in \cite{gelbrich1990formula}, a universal lower bound for any pair  $(\IP,\IQ)$ having the same moments (expectation and covariance) as the 
Gaussian law, see  also \cite{Cuesta1996lower}.


\paragraph{Limit laws for the Gaussian Wasserstein distance}

More specifically, from now on let the laws $\P,\Q \in \mcM_1(\R^d)$ be in the class of $d$-variate normals, i.e.
\begin{equation}\label{eq.P.Q}
 \IP \sim N(\mu, \Sigma) \text{ and }\IQ \sim N(\nu, \Xi)\text{ for some }\mu, \nu \in  \IR^{\dimension}, \, \Sigma, \Xi \in S_+(\IR^\dimension), 
\end{equation}
the  symmetric, positive definite, $\dimension$-dimensional matrices.
From now on we will also restrict to $p=2$.
 In this case  the Wasserstein distance between $N(\mu,\Sigma)$ and $N(\nu,\Xi)$ is computed as (see \cite{dowson1982frechet, olkin1982distance, GS:84})
 \begin{equation}\label{eq:was:dist:Gauss} 
 \Gw := \wass_2^2(\IP,\IQ) = \|\mu-\nu\|^2 + \tr(\Sigma) + \tr(\Xi) - 2\tr \left[ \left[\Sigma^{1/2} \Xi \Sigma^{1/2} \right]^{1/2} \right] \, .
 \end{equation}
Here, $\tr$ refers to the trace of a matrix and its square root  is defined in the usual spectral way.
The norm $\| \cdot \|$ is the Euclidean norm with corresponding scalar product denoted by $\langle \cdot, \cdot \rangle.$
Now, if we replace $\IP$ with the empirical measure $\IP_n$ and read $\mu$ and $\Sigma$ as a functional of $\P$, we obtain the empirical Wasserstein estimator $\Gwhat_n$ restricted to the $d$-dimensional Gaussian measures as
\begin{equation}\label{e.estimator}
  \begin{split}
  \Gwhat_n  & = \Gwhat_n (X_1,\dotsc, X_n, \Q)\\
  :&=  \wass_2^2\left(N (\hat{\mu}_n, \hat{\Sigma}_n),N(\nu, \Xi)\right )\\
  & = \|\hat \mu-\nu\|^2 + \tr(\hat \Sigma) + \tr(\Xi) - 2\tr \left[ \left[\hat \Sigma^{1/2} \Xi \hat \Sigma^{1/2} \right]^{1/2} \right] .
  \end{split}
\end{equation}
Similar to the case of the general empirical Wasserstein distance for $d=1$ we find in the following that the asymptotic behavior differs whether $\P = \Q$, i.e.~$\mu=\nu$ and $\Sigma = \Xi$ or $\P \neq \Q$.
Let us start with the latter case which turns out to be simpler.
We show in Theorem \ref{T.WS.ASYMP.ONESAMPLE}, whenever 
$\P \neq \Q$, i.e. $\mu \neq \nu$ or $\Sigma \neq \Xi$  a limit theorem as in (\ref{scalinglimit}) holds with $a_n = n^{1/2}$ and $b_n=0$, i.e.~as $n \rightarrow \infty,$
  \begin{equation}
 \label{eq:onesample.INT}
 \sqrt{n}\left(\wass_2^2\left(N (\hat{\mu}_n, \hat{\Sigma}_n),\IQ\right) - \wass_2^2(\IP, \IQ)\right) \Rightarrow N(0, \upsilon^2) .
 \end{equation}
Here the asymptotic variance can be explicitly computed as
\begin{align}
  \label{eq:varonesample.INT}
  \upsilon^2 &= 4 ( \nu - \mu)^t \Sigma (\nu-\mu) +  2 \tr (\Sigma^2 + \Sigma \Xi) - 4 \sum_{k=1}^d \kappa_k^{1/2}r_k^t \Xi^{-1/2} \Sigma \Xi^{1/2} r_k,
  \end{align}
where $\{(\kappa_k,r_k):\, k=1,\dotsc, d\}$ denotes the eigendecomposition of the symmetric matrix $\Xi^{1/2}\Sigma\Xi^{1/2}$ into orthonormal eigenpairs (consisting of eigenvalues and eigenvectors).
Here and in the following we denote by ${}^t$ the transpose of a vector (or matrix).
We will also treat the two sample case (Theorem \ref{T.WS.ASYMP}), where $\IQ$ is additionally estimated under the Gaussian restriction by a second independent sample.
In this case,  the eigendecomposition of $\Sigma$ itself given by $\{(\lambda_l,p_l):\, l=1,\dotsc, d\}$ additionally occurs in the limiting variance, which disappears for distinct eigenvalues of $\Sigma$, however.

\smallskip

Our proof relies on the Fr\'echet differentiability of the Wasserstein-distance in the Gaussian  case (see Theorem \ref{p:Phi:diff}) together with a Delta method (Theorem \ref{thm:vdv:delta}).
The formula for the Gaussian Wasserstein distance \eqref{eq:was:dist:Gauss} can be seen as a (non-linear) functional of symmetric operators.
The proof of its Fr\'echet differentiability is based on the second order perturbation of a general compact Hermitian operator (see Corollary \ref{cor:Second derivative}); this result is of interest on its own.
In a similar way we treat the case $\IP = \IQ$. Here, the first derivative vanishes and we show that the asymptotic distribution is determined by the Fr\'echet derivative of second order.
This gives for  $a_n = n$ and $b_n=0$ a non-degenerate limit which can be characterized as a quadratic functional of a Gaussian r.v., see Theorem \ref{T.WS.ASYMP:P=Q}.
Note that all scaling rates are independent of the dimension $d$ ($d$ enters in the constants, though) and coincide with those for dimension $d=1$ for the general empirical Wasserstein distance based on \eqref{explonedimension}.


\paragraph{Comparison to known results in $d\geq 2$}

Although distributional results of $\widehat{\wass}_{p,n}$ are not known for $d\geq 2$ it is illustrative to discuss our
limit results in the light of some known results on bounds of the moments of $\widehat{\wass}_{p,n}$ and a.s.~limits.
We will restrict to $p=2$, since our results apply only to that case.

A particularly well understood case for $\P = \Q$  is the uniform distribution on the $d$-dimensional hypercube, see \cite{AKT:84}, \cite{talagrand_matching_1992} and \cite{dobric_asymptotics_1995}.
In \cite{dobric_asymptotics_1995} it is shown that $c_n \wass_2^2(\P_n,\P) \to \lambda$ almost surely for a certain $\lambda \in (0,\infty)$ and $c_n = n^{2/d}$ when $d\geq 3$.
To the best of our knowledge a finer distributional asymptotics in the sense of 
\begin{equation}\label{e.tr5}
 a_n( \wass_2^2(\P_n,\P) + b_n) \Rightarrow Z
\end{equation}
for $b_n = - \lambda/c_n$ with non-degenerated r.v.~$Z$ is not known.
If it existed, it would require $c_n = o(a_n)$.
Our Theorem \ref{T.WS.ASYMP:P=Q} affirms the existence of the limit in \eqref{e.tr5} for the Gaussian Wasserstein estimator with $a_n = n$ and $b_n=0$.
The fact that $a_n= n$ grows faster than $c_n = n^{2/d}$ for $d \geq 3$ was expected by the previous argument.
Similar argumentation holds for the result in \cite{AKT:84} where upper and lower almost sure bounds $\lambda_1 < \lambda_2 \in (0,\infty)$ are given for $c_n = n/(\log n)^2$.
To subsume the comparison to the almost sure results: our rate $a_n=n$ is in the range of possible rates, i.e.~$1/a_n = o(1/c_n)$, which may be expected for non-trivial distributional limit results.
Recall, however, that we are not proving a limit as in \eqref{e.tr5}, but in the sense of \eqref{eq:onesample.INT}, where $\wass_2^2(\P_n,\Q)$ is replaced by $\Gwhat_n$.

It is also interesting to compare our rate for $\P = \Q$ to moment bounds.
In \cite{fournier_rate_2013} upper bounds are given for $E[\wass_2^2(\P_n,\P)]$  when $\P$ is a measure with finite moments of any order (recall that Gaussian distributions have moments of any order).
They obtain that $d_n E[\wass_2^2(\P,\P)] \leq C$ for a constant $C < \infty$ if $d_n=n^{2/d}$ for $d\geq 5$, $d_n = n^{1/2}$ for $d \leq 3$ and $d_n= n^{1/2} (\log (1+n))^{-1}$ for $d = 4$.
All those results are consistent with our result in the sense that $\limsup_{n \to \infty} d_n/a_n  < \infty$.

So far we have only discussed the case $\P = \Q$.
The literature on the case $\P \neq \Q$ is much scarcer.
For the case $d=1$, \cite{munk_nonparametric_1998} obtain in situations comparable to ours also asymptotical normality for $a_n = n^{1/2}$ and $b_n = 0$.
This is the same scaling rate as the one observed in our case.
For higher dimensions we do not know of any explicit results.
Theorem 3.9 in \cite{dBM13-2} gives a result which is similar to \eqref{eq:onesample.INT} except that their setting deals with an incomplete transport problem.
Their rate for $d\geq 1$ is also $a_n = n^{1/2}$ and they obtain that the left hand side of the corresponding version of \eqref{eq:onesample.INT} is bounded in probability.
In particular, they do not state an explicit limit law as we can give it in our special situation.

To the best of our knowledge, other results are not yet available for the case $\P \neq \Q$ in higher dimensions.

\paragraph{Elliptical distributions}

It is possible to  generalizate the result in \eqref{eq:onesample.INT} beyond the class of Gaussian distributions.
As \cite{gelbrich1990formula} showed in Theorems 2.1 and 2.4, formula \eqref{eq:was:dist:Gauss} holds for more general classes of distributions (with appropriate modifications); i.e.~elliptically symmetric probability measures.
We comment on this generalization in Remark \ref{r.gelbrich}.

\paragraph{Statistical Applications: Inference and Bootstrap}

 A bootstrap limiting result follows immediately from our proof as well. For the case of $\IP$ and $\IQ$ being different the first order term  in the Fr\'echet expansion determines the asymptotics and an $n$ out of $n$ bootstrap is valid  (see e.g.~\cite{vdV:98}). Other resampling schemes, such as a parametric bootstrap can be applied as well (see e.g.~\cite{shao1995jackknife}). When $\IP = \IQ$ the second order Fr\'echet derivative matters and one has to resample fewer than $n$ observations, i.e.~$o(n)$ (see e.g.~\cite{bickel1997resampling})  
to obtain bootstrap consistency. As the limiting laws are rather complicated, bootstrap seems to be a reasonable option for practical purposes, e.g.~confidence intervals for the Wasserstein distance in the Gaussian case can be obtained from this \cite{shao1995jackknife, davison1997bootstrap}. We provide more details on 
bootstrapping the Gaussian Wasserstein distance and an application to structure determination of proteins in Section \ref{ssection:Applications}.

\smallskip

\smallskip

The paper is organized as follows. In Section \ref{ss.limits}, we present the main results on the asymptotic distribution in the one and two sample case  and in Section~\ref{Sec:Frechetdiff} we provide the main results on Fr\'echet differentiability of the underlying functional. Next, we give two applications, one theoretical application regarding the bootstrap in Section~\ref{ssection:Applications} and a practical one regarding a data example for the positions of amino acids in a protein in Section~\ref{s.app}. In Section~\ref{s.2} we present the proofs of the main theorems.
Finally, the appendix comprises some required facts and technical results on functional differentiation.
  
\section{Main Results}\label{s.results}


\subsection{Limit laws for the Gaussian Wasserstein distance}\label{ss.limits}

This section contains the three main results on convergence of the empirical Wasserstein distance estimator $\Gwhat_n$  defined in \eqref{e.estimator}.
The first two theorems present the case where $\P \neq \Q$ and the last one states the result for $\P = \Q$.

\medskip

Suppose now that $\P \neq \Q$ are both Gaussian distributions as in \eqref{eq.P.Q}.
We denote their Wasserstein distance by $\Gw := \Gw(\P,\Q)$.
Suppose we have independent samples from these different distributions.
Then we obtain the following result.
\begin{theorem}[Asymptotics for the empirical Wasserstein distance in the one sample case, $\P \neq \Q$]\label{T.WS.ASYMP.ONESAMPLE}
 Let $\IP\neq \IQ$ in $\mathcal{M}_1(\IR^d)$ be Gaussian, $\IP \sim N(\mu,\Sigma), \IQ \sim N(\nu,\Xi)$ with $\Sigma$ and $\Xi$ having full rank.
 Let $X_1,\dotsc, X_n \stackrel{i.i.d.}{\sim} N(\mu,\Sigma)$ and consider the Gaussian Wasserstein estimator $\Gwhat_n$ from \eqref{e.estimator}.
 Then as $n \rightarrow \infty,$
  \begin{equation}
 \label{eq:onesample}
 \sqrt{n}\left(\Gwhat_n - \Gw  \right) \Rightarrow N(0, \upsilon^2)
 \end{equation}
 where 
 \begin{equation}\label{eq:varonesample}
  \begin{split}
  \upsilon^2 = 4 ( \nu - \mu)^t \Sigma (\nu-\mu) +  2 \tr &\left(\Sigma^2\right) + 2 \tr\left(\Sigma \Xi\right) \\
   &- 4 \sum_{k=1}^d \kappa_k^{1/2}r_k^t \Xi^{-1/2} \Sigma \Xi^{1/2} r_k.
 \end{split} \end{equation}
 Here, $\{(\kappa_k,r_k):\, k=1,\dotsc, d\}$ denotes the eigendecomposition of the symmetric matrix $\Xi^{1/2}\Sigma\Xi^{1/2}$ into orthonormal eigenpairs (consisting of eigenvalues and   eigenvectors).
\end{theorem}

In many practical applications we may not have direct access to the parameters of the distribution $\Q = N(\nu,\Xi)$ and we merely have a sample from that distribution.
The generalization of the estimator from \eqref{e.estimator} for the two sample case is given as
\begin{equation}\label{e.estimator.2}
 \Gwhat_{n,m} (X_1,\dotsc, X_n, Y_1,\dotsc, Y_m) :=  \wass_2^2\left(N (\hat{\mu}_n, \hat{\Sigma}_n),N(\hat{\nu}_m, \hat{\Xi}_m) \right) \, .
\end{equation}

The following result is the two sample analogue to Theorem \ref{T.WS.ASYMP.ONESAMPLE}.
\begin{theorem}[Asymptotics for the empirical Wasserstein distance in two sample case, $\P \neq \Q$]\label{T.WS.ASYMP}
 Let $\IP\neq \IQ$ in $\mathcal{M}_1(\IR^d)$ be Gaussian, $\IP \sim N(\mu,\Sigma), \IQ \sim N(\nu,\Xi)$ with $\Sigma$ and $\Xi$ having full rank.
 Let $n \in \IN$ and $m=m(n) \in \IN$ such that $n/m(n) \to a/(1-a)$ as $n \to \infty$ for a certain $a \in (0,1)$.
 Consider the i.i.d.~samples $(X_1,\dotsc, X_n)$ and $(Y_1,\dotsc,Y_m)$ with joint law $\IP$ and $\IQ$, respectively.
 Then as $n \rightarrow \infty,$
 \begin{equation}
  \label{eq:twosample.m.n}
    \sqrt{\frac{mn}{m+n}}\left(  \Gwhat_{n,m} -  \Gw  \right) \Rightarrow N(0,\varpi^2),
 \end{equation}
 where
 \begin{align}
 \label{eq:vartwosample.m.n}
    \varpi^2 &= 4( \nu - \mu)^t ((1-a)\Sigma+a\Xi) (\nu-\mu)+   2 \tr ((1-a)\Sigma^2+a\Xi^2)  \\
    \nonumber
    &   \quad \quad \quad + 2a \tr(\Sigma \Xi) - 4  \sum_{k=1}^d \kappa_k^{1/2}q_k^t ((1-a)\Sigma+a\Sigma^{-1/2}\Xi\Sigma^{1/2}) q_k \, \\
     \nonumber
    & \quad \quad \quad \quad
    - 2 (1-a)\sum_{k,l=1}^d \kappa_l^{1/2} \kappa_k^{1/2} \sum_{i=1}^d \sum_{\stackrel{j=1}{j \neq i, \lambda_i =\lambda_j}}^d q_l^tp_i p_i^tq_k q_l^t p_j p_j^t q_k  .
  \end{align}
  Here, $\{(\kappa_l,q_l):\, l=1,\dotsc, d\}$ denotes the eigendecomposition of the symmetric matrix $\Sigma^{1/2}\Xi\Sigma^{1/2}$ into orthonormal eigenpairs (consisting of eigenvalues and   eigenvectors)  and  $\{(\lambda_l,p_l):\, l=1,\dotsc, d\}$ is the eigendecomposition of $\Sigma$.
\end{theorem}
  
\begin{remark}[Distinct eigenvalues of $\Sigma$]
We note that the last term of (\ref{eq:vartwosample.m.n}) disappears if all eigenvalues $\lambda_l$, $l=1, \dots, d$ of $\Sigma$ are distinct. 
\end{remark}   

\begin{remark}[Commutative case $\Sigma \Xi =\Xi \Sigma$]
In this case, we can choose the eigenbasis of $\Sigma$ and $\Xi$ to be the same, which is thus also an eigenbasis of $\Sigma^{1/2}\Xi \Sigma^{1/2}$ implying that $p_i^t q_l = \delta_{il}.$
Denoting by $\lambda_k'$ the eigenvalues of $\Xi$ we have that $\lambda_k \lambda_k' = \kappa_k.$
Using this, we see that the last term of (\ref{eq:vartwosample.m.n}) disappears.
For $a=1/2$ the last term of  (\ref{eq:varonesample})  and the second and third last term of (\ref{eq:vartwosample.m.n}) simplify to $-4 \tr( (\Sigma^{1/2} \Xi\Sigma^{1/2})^{1/2} \Sigma)$ and $-4 \tr( (\Sigma^{1/2} \Xi\Sigma^{1/2})^{1/2} (\Sigma + \Xi))$, respectively. 
Thus, in this case for our two previous theorems the following simplifications apply
 \begin{eqnarray*}
  \upsilon^2 &=& 4 ( \nu - \mu)^t \Sigma (\nu-\mu) +  2\tr\left(\Sigma (\Sigma^{1/2} -\Xi^{1/2})^2\right), \\
    \varpi^2 &=& 4( \nu - \mu)^t (\Sigma+\Xi) (\nu-\mu)+ 2\tr\left(\Sigma (\Sigma^{1/2} -\Xi^{1/2})^2\right) \\
    & & \phantom{AAAAAAAAAAAAAAAAA}+ 2\tr\left(\Xi (\Xi^{1/2} -\Sigma^{1/2})^2\right).
 \end{eqnarray*}
\end{remark}
The result is restricted to the case of covariance matrices of full rank.
We comment on that restriction in Remark \ref{r.tr1}.
 
Note in the previous result that the variance $\varpi$ is zero if $\IP=\IQ$, i.e.~$\mu=\nu$ and $\Sigma=\Xi,$ see also Remark \ref{remark:derivcomm}.
In this case a second order expansion provides a valid limit law.
Namely, the following theorem holds.
 \begin{theorem}[Asymptotics for the empirical Wasserstein distance, $\P = \Q$]\label{T.WS.ASYMP:P=Q}
 Let $\IP \in \mathcal{M}_1(\IR^d)$ be Gaussian: $\IP \sim N(\mu,\Sigma)$ with $\Sigma$ having full rank.
 Consider the i.i.d.~samples $(X_1^{(1)},\dotsc, X_n^{(1)})$ and $(X_1^{(2)},\dotsc, X_n^{(2)})$ with joint law $\IP$ and the Gaussian Wasserstein estimators $\Gwhat_n$ from \eqref{e.estimator} and $\Gwhat_{n,n}$ from \eqref{e.estimator.2}.
 Then as $n \rightarrow \infty,$
 \begin{equation}\label{e.P=Q.onesample}
   n \Gwhat_{n} \Rightarrow Z_1 \, ,
 \end{equation}
 and
 \begin{equation}
 \label{eq:P=Q}
 n \Gwhat_{n,n} \Rightarrow Z_2 \, ,
 \end{equation}
 where $Z_1$ and $Z_2$ are random variables, characterized in \eqref{e.tr2}.
 \end{theorem}

\begin{remark}[The limiting distribution for $\P = \Q$]
 An explicit description of the limit in \eqref{e.tr2} is difficult in general, a simplification can be given in the one dimensional case, see \eqref{e.tr3}.
  
 Analogously to Theorem \ref{T.WS.ASYMP} it is also possible to obtain the asymptotics in the case that the sample sizes for $\IP$ and $\IQ$ are not the same.
 In the regime $n/m \to a/(1-a)$ for a certain $a\in (0,1)$ as $n \to \infty$ the convergence holds for $\left((nm)/(n+m) \right) \Gwhat_{n,m}$.
\end{remark}

\begin{remark}[Generalization to elliptically symmetric distributions]\label{r.gelbrich}
Gelbrich \cite{gelbrich1990formula} showed that formula \eqref{eq:was:dist:Gauss} holds for any two elements $\P, \Q \in \mcM_1(\R^d)$ which are translations of distributions whose covariance matrices are related in a certain way. He also showed that this condition is fulfilled as long as they are in the same class of elliptically symmetric distributions. The class of Gaussian distributions is such a class. 

More generally, denote by $S^0_+(\R^d)$ the non-negative definite, symmetric matrices and by  
$\text{rk}_A$ the  rank of any $A \in S^0_+(\R^d).$ Let $f: [0,\infty) \rightarrow [0,\infty)$ be a measurable function
 that is not almost everywhere zero and that satisfies
\begin{equation}
\label{ellipclassf}
\int_{-\infty}^{\infty} |t|^l f(t^2) dt < \infty, \quad l=d-1,d,d+1,
\end{equation}
and set $c_A=\int f( \langle x, A x \rangle) dx.$
Then, one can consider classes of the form 
\begin{align}
\label{ellipclass}
\mcM_1^f(\R^d)&:=\{ \IP \in \mcM_1(\text{Im}_A) :  A \in S^0_+(\R^d) \text{ with rk}_A=d, \IP  \text{ has density }\\
\nonumber
&  \quad \quad  f_{A,v}(x) = c_A f( \langle x-v, A (x-v) \rangle), x \in \R^d,  v \in \R^d\},
\end{align}
see Theorem 2.4 of \cite{gelbrich1990formula} (there stated also for matrices $A$ that do not have full rank, which is not considered here).  

As can be seen from \eqref{ellipclassf} and \eqref{ellipclass}  by setting $f=1_{[0,1]}$  
another prominent example for elliptically symmetric distributions is that of uniformly distributed probability measures on ellipsoids, i.e.~on sets of the form $U_{A,v}:=\{x \in \R^d: \, \langle x-v, A (x-v) \rangle \leq 1\}$ for $A \in S^0_+(\R^d)$ and $v \in \R^d.$ Furthermore, we obtain the multivariate t-distributions (with $\nu>0$ degrees of freedom) by setting $f(t)= (1+\frac{t}{\nu})^{-(\nu+d)/2}.$ These play a particular role for copulae models, see \cite{Demarta}.

As the largest part of the proofs of Theorems \ref{T.WS.ASYMP.ONESAMPLE}, \ref{T.WS.ASYMP} and \ref{T.WS.ASYMP:P=Q} relies only on the specific form of formula \eqref{eq:was:dist:Gauss}, the results of these theorems immediately transfer to other classes of elliptically symmetric distributions. What still needs to be verified in the various cases is that a central limit theorem holds for the empirical mean and for the covariance matrices, see our Lemma \ref{lem:anderson} in the Gaussian case.
For example, this requires $\nu \geq 2$ for the class of multivariate $t$-distributions to guarantee the existence of second moments.
The specific form of the analogous limits in Theorems \ref{T.WS.ASYMP.ONESAMPLE}, \ref{T.WS.ASYMP} and \ref{T.WS.ASYMP:P=Q} will depend on the specific form of the limit in the appropriate central limit theorem and has to be computed from case to case.
\end{remark}

\subsection{Fr\'echet differentiability of the Gaussian Wasserstein distance}\label{Sec:Frechetdiff}

The concept of differentiation on Banach spaces will be an important tool for the proof of the results in the previous section.
We give a comprehensive reminder of some classical results for Fr\'echet derivatives in Section \ref{s.Functional.deriv}.
Moreover some more advanced results about a Taylor expansion of a functional of an operator may be found in Section \ref{s.2nd.order}.

Now we consider the 2-Wasserstein distance of Gaussian distributions as a functional of their means
and covariance matrices (see (\ref{eq:was:dist:Gauss})).
In the following we show its \Frechet{} differentiability and explicitly derive its \Frechet{} derivative.
To this end, consider $A,B \in S_{+}(\Rdim)\subset L(\Rdim, \Rdim)\simeq \R^{d \times d}$ (symmetric, positive definite matrices).
We use the eigenvalue decomposition for $A$ and $A^{1/2}BA^{1/2}$ of the form 
\begin{equation}\label{eq:lambda,kappa}\begin{split}
 A & = \sum_{i=1}^{\dimension} \lambda_i P_i, \\
 A^{1/2}BA^{1/2} & = \sum_{i=1}^{\dimension} \kappa_iQ_i ,
\end{split}\end{equation}
where $\lambda_i, \kappa_i >0$, $P_iP_{j} = \delta_{ij}P_i, Q_iQ_{j}=\delta_{ij}Q_i , 1\leq i\leq j \leq d.$
Our decomposition implies that all projections $P_i$, $Q_j$ are onto one dimensional spaces such that we can write $Q_i = q_iq_i^t$ and $P_i =p_i p_i^t$ for vectors $q_i$ and $p_i$ in $\Rdim, i=1,\dots,d.$

\begin{lemma}[Differentiability of the $2$-Wasserstein distance $\wass_2^2$ of Gaussian distributions]\label{p:Phi:diff}
Let $\Phi: \R^{2d} \times S_{+}(\Rdim)^2 \to \R$ be given by
\begin{equation}
\label{Phi}
 (\mu,\nu,A,B) \mapsto \|\mu -\nu\|^2 + \tr(A) + \tr(B) - 2 \tr\left( (A^{1/2}BA^{1/2})^{1/2} \right). 
 \end{equation}
This mapping is \Frechet{} differentiable and its derivative at $(\mu,\nu,A,B) \in \R^{2d} \times S_{+}(\Rdim)^2$ is a mapping in $L(\R^{2d} \times \R^{2(d \times d)}, \R)$ given by
\begin{equation}
 \begin{split}
  D_{(\mu,\nu,A,B)}\Phi &  [(g,g', G,G')]   =  2 (\mu-\nu)\cdot (g-g')  + \tr G + \tr G'  \\
 \label{eq:D} &  - \sum_{l=1}^\dimension \kappa_l^{1/2} \sum_{i=1}^\dimension  \lambda_i^{-1} q_l^t P_i G P_i q_l - \sum_{l=1}^\dimension \kappa_l^{-1/2} q_l^t \sqrt{A} G' \sqrt{A} q_l  \\
  & \quad -  \sum_{l=1}^\dimension \kappa_l^{1/2} \sum_{\substack{ i, m=1\\ \lambda_i \neq \lambda_m}}^\dimension \left( \lambda_i  \lambda_m\right)^{-1/2}  q_l^t P_i G P_m  q_l
 \end{split}
\end{equation}
for all $g,g' \in \R^{d}$, $G, G' \in \R^{d \times d}$ and with $\kappa, \lambda, P, Q$ as in \eqref{eq:lambda,kappa}.
\end{lemma}

\begin{remark}
 Note that the last result is stated in finite dimensional spaces.
 Obviously in this case \Frechet{} differentiability coincides with usual differentiability. 
 Nonetheless, we prefer to use the abstract setup for simpler notation, obvious extensions to the infinite-dimensional case and because it is consistent with the cited references.
\end{remark}

Recall that $\Phi$ is a symmetric function in the entries $\mu$ and $\nu$ and likewise in $A$ and $B$.
If we switch the notation in the previous theorem and then consider $g'$ and $G'$ equal to zero we obtain as an immediate consequence.
\begin{corollary}
\label{p:Phi:diff-onesample}
Let $\Phi^{(\nu,B)}: \R^{d} \times S_{+}(\Rdim) \to \R$ be given by $\Phi$ from (\ref{Phi}) as a function of $\mu$ and $A$ for fixed $\nu \in \R^{d}$ and $B \in S_{+}(\Rdim).$
Then $\Phi^{(\nu,B)}$ is \Frechet{} differentiable and its derivative at any point $(\mu,A) \in \R^{d} \times S_{+}(\Rdim)$ is an element of $L(\R^{d} \times \R^{d \times d}, \R)$ given by
\begin{align} \label{eq:D:onesample}
 D_{(\mu,A)}\Phi^{(\nu,B)} [(g, G)]   = & 2 (\mu-\nu)\cdot g  + \tr G - \sum_{l=1}^d \kappa_l^{-1/2}r_l^t\sqrt{B}G \sqrt{B} r_l \, ,
\end{align}
for all $g \in \R^{d}$, $G \in \R^{d \times d}$ and $\{(\kappa_l,r_l),l=1,\dotsc,d\}$ the eigendecomposition of $B^{1/2}AB^{1/2}$ as in \eqref{eq:lambda,kappa}.
\end{corollary}

The previous theorem also allows a simpler representation of the derivative if we restrict to certain cases.
\begin{remark}[Commutative case $AB =BA$] 
\label{remark:derivcomm}
Here, we can choose the eigenbasis of $A$ and $B$ to be the same, which is thus also an eigenbasis of $A^{1/2}BA^{1/2}$ implying that $P_i q_l = \delta_{il}q_l$: If $\lambda_k'$ are the eigenvalues of $B$ this implies that  $\lambda_k \lambda_k' = \kappa_k$ and we obtain in Proposition \ref{p:Phi:diff}
 \begin{equation}\label{eq:D:symm}
  \begin{split}
   D_{(A,B)} & \phi^{(2)} [(G, G')] 
 = \tr G + \tr G' -\sum_{l=1}^\dimension \kappa_l^{1/2} \lambda_l^{-1}q_l^t G q_l - \kappa_l^{-1/2} \lambda_l q_l^t G' q_l \\
&= \sum_{l=1}^\dimension q_l^t G q_l \left(1-(\frac{\lambda_l'}{\lambda_l})^{1/2}\right) - \sum_{l=1}^\dimension q_l^t G' q_l \left(1-(\frac{\lambda_l}{\lambda_l'})^{1/2}\right).
  \end{split}
\end{equation}

This implies in particular that the derivative equals zero iff $A=B$.
\end{remark}

At the end of this section we state a result on the second order derivative of~$\Phi.$
\begin{theorem}\label{p:Phi:diff:2}
Let $\Phi: \R^{2d} \times S_{+}(\Rdim)^2 \to \R$ be as in Proposition \ref{p:Phi:diff}. The mapping is twice \Frechet{} differentiable. Its second derivative $D^2_{(\mu,\nu,A,B)}\Phi$ at a point $(\mu,\nu,A,B) \in \R^{2d} \times S_{+}(\Rdim)^2$ is a symmetric bilinear mapping from $\R^{2d} \times \R^{2(d \times d)} \times \R^{2d} \times \R^{2(d \times d)} \to \R$ which is defined by its quadratic form $$D^2_{(\mu,\nu,A,B)}\Phi [(g,g',G,G'),(g,g',G,G')], \quad 
g,g' \in \R^{d}, G, G' \in \R^{d \times d}.$$
\end{theorem}
\begin{remark}
 It would be possible to use the calculations from Corollary \ref{cor:Second derivative} in Theorem \ref{T.WS.ASYMP:P=Q} to obtain an explicit formula for the second derivative for $d\geq 2$.
 However, this calculation is very tedious even for $d=2$ and we will not carry it out here.
\end{remark}

\subsection{Bootstrap}
\label{ssection:Applications}

Applications of Theorems~\ref{T.WS.ASYMP.ONESAMPLE} and \ref{T.WS.ASYMP} such as the construction of confidence sets require one to estimate the variances $\upsilon$ and $\varpi$ of the limiting distribution.
But for the construction of confidence sets those quantities need to be estimated.
Of course, these can be estimated from the data by their empirical counterparts  as well.
Another option is to bootstrap the limiting distribution, which becomes particularly useful for an application of Theorem~\ref{T.WS.ASYMP:P=Q} as the limiting distribution has a complicated form.
In fact, due to the differentiability results of the last section (Lemma~\ref{p:Phi:diff} and Theorem~\ref{p:Phi:diff:2}) we can rigorously establish such a bootstrap.
We illustrate the bootstrap approximation for the one sample case, the two sample case is analogous.
For $m \leq n$ we denote by $X_1^\ast, \dots X_m^\ast$ an independent resampling (with replacement) of the sample $X_1,\dots,X_n$ and define $\Gwhat_m^\ast$ as in 
(\ref{e.estimator}) using that resampling.  

As in the beginning of Section~\ref{ss.limits} a distinction for the cases $\P \neq \Q$ and $\P = \Q$ is required.
The former allows an $n$-out-of-$n$ bootstrap, the latter requires an $m$-out-of-$n$ bootstrap, s.t.~$m=o(n)$.

\begin{proposition}[$n$ out of $n$ bootstrap]
 \label{propn:onesample.bootstrap}
 Suppose $\P \neq \Q$.
 Then 
   \begin{equation}
 \label{eq:onesample.bootstrap}
 \sqrt{n}\left(\Gwhat_n^\ast - \Gwhat_n  \right) \Rightarrow N(0, \upsilon^2)
 \end{equation}
 conditionally given $X_1,X_2,\dots$ in probability.
\end{proposition}
Here, weak convergence conditionally given $X_1,X_2,\dots$ in probability means the following: Denote by $\rho$ a metric corresponding to the topology of weak convergence and by ${\cal L}(\cdot)$ the law of a random quantity. Then (\ref{eq:onesample.bootstrap}) means that $\rho({\cal L}(\Gwhat_n^\ast - \Gwhat_n)), N(0, \upsilon^2))$ as a function of $X_1, \dots, X_n$ converges to zero in probability. 

Proposition \ref{propn:onesample.bootstrap} follows immediately from Theorem 23.5 in \cite{vdV:98} combined with the differentiability of Lemma~\ref{p:Phi:diff} and the strong consistency of the bootstrap result for the sample mean and the sample covariance matrix of Gaussian distributions, respectively.
Note that this follows from the bootstrap consistency of the multivariate empirical process (Theorem~23.7 in \cite{vdV:98}) together with Hadamard differentiability of $\Sigma(F)$.
In our case this also follows immediately in an elementary way from the fact that $\hat{\Sigma}_n$
is independent of $\hat{\mu}_n$ and from the fact that its distribution
does not depend on $\mu.$ Thus, it follows from the bootstrap consistency for the multivariate i.i.d.~average $\frac{1}{n} \sum_{i=1}^n X_i X_i^t$.

Note that since the left hand side of (\ref{eq:onesample.bootstrap}) only depends on the sample (and some further randomness) the result serves to estimate the right hand side and so in particular $\upsilon^2.$

For $\P = \Q$ we obtain the $m$ out of $n$ bootstrap.
\begin{proposition}[$m$ out of $n$ bootstrap]
Let $m=m(n)$ such that $m(n)/n \rightarrow 0$ as $n \rightarrow \infty.$  Suppose further that $\P = \Q$.
 Then 
   \begin{equation}
 \label{eq:onesample.bootstrap.P=Q}
 m\left(\Gwhat_m^\ast - \Gwhat_n  \right) \Rightarrow Z_1 ,
 \end{equation}
  conditionally given $X_1,X_2,\dots$ in probability, where $Z_1$ is the distribution from Theorem~\ref{T.WS.ASYMP:P=Q}.
\end{proposition}
This follows along the lines of the proof of Theorem~5.1 in \cite{freitag2005hadamard} using the second order differentiability of Theorem~\ref{p:Phi:diff:2} together with the  $m$ out of $n$ bootstrap consistency result for the sample mean and sample covariance matrix of Gaussian distributions.

\section{Applications}\label{s.app}

Theorems~\ref{T.WS.ASYMP.ONESAMPLE}-\ref{T.WS.ASYMP:P=Q} can be used (in combination with the bootstrap results in Section~\ref{ssection:Applications}) for several purposes: e.g.~testing the null hypotheses $H: \Gw = 0$ (Theorem~\ref{T.WS.ASYMP:P=Q} for the two sample case) or neighborhood hypotheses of the form $H: \Gw > \delta$ vs.~$K: \Gw \leq \delta$ in order to validate the closeness of the multivariate normal distributions in Wasserstein distance.
Here $\delta >0$ is a threshold to be fixed in advance, see e.g.~\cite{CM:98} for a related test and further references on these types of testing problems. The test amounts to rejecting whenever $r_n(\Gwhat - \delta) > u_\alpha$ (see Theorem~\ref{T.WS.ASYMP.ONESAMPLE} in the one sample case), where $u_\alpha$ denotes the $\alpha$-quantile of a standard normal random variable.

Another immediate consequence are (bootstrap) confidence intervals for $\Gw$; which require the asymptotics for $\P \neq \Q$ (Theorem~\ref{T.WS.ASYMP.ONESAMPLE} and Theorem~\ref{T.WS.ASYMP}), see \cite{shao1995jackknife} for a general exposition and \cite{czado_assessing_1998} for the Wasserstein distance on the real line.
In the following  we  exemplarily illustrate our methodology for the one sample test on a real data application.

\subsection{Positions of  amino acids in a protein}

In order to understand the biological function of proteins it is important to know both their three dimensional structure as well as their conformational dynamics.
X-ray crystallography and NMR spectroscopy can help elucidate high-resolution information about biomolecular structures but conformational dynamics is more elusive, see e.g. \cite{Honndorf:12}.
Small-amplitude dynamics is thought to be reflected by crystallographic B factors, whereas NMR structures are often interpreted as native state ensembles. However, both interpretations should be taken with some caution. Therefore, it is of interest to investigate whether the crystallographic view on conformational dynamics provided by B factors agrees with the ensemble view provided by NMR. To this end, we will use the Wasserstein based test  to quantify to what amount the local flexibility measured by X-ray crystallography agrees with the structural variability seen by NMR.

\smallskip

Our analysis is based on the crystallographic model of proteins, see Section~2.2 of \cite{Trueblood:es0238}. This model postulates that each amino acid in the protein is a point that has a position which follows a Gaussian distribution with mean $\mu \in \R^3$ and covariance matrix $\beta^2 \1$, where $\1$ is the identity matrix in $\R^3$.
It is customary to assume the positions of different amino acids as independent. In this model, the quantity $\beta^2$ is called the $B$-factor and is related to the Debye-Waller factor used in crystallography.

We focus on a particular amino acid and want to compare the ``true'' distribution $\P$ proposed by the crystallographic model to the samples obtained from NMR spectroscopy. Arguably, the  Wasserstein distance is particularly well-suited for this scenario as it accounts for a measurement of displacement.

In order to obtain a  test for the hypothesis
\begin{equation}\label{e.H_0}
 \text{H}_0:\ \text{the samples come from the true (Gaussian) distribution } \P .
\end{equation}
we apply Theorem~\ref{T.WS.ASYMP:P=Q}. In the present setting 
 we can further simplify and  obtain  an explicit description of the limit law $Z_1$.
If we assume that $\P = N(\mu,\sigma^2 \1)$ as reference distribution in $\R^3$, then we obtain for \eqref{e.P=Q.onesample}:
\begin{equation*} 
 Z_1 = \sigma^2 (2X + 6 X' + \tfrac{3}{2} X'')
\end{equation*}
for independent $\chi^2_3$-random variables $X$ and $X'$ and $\chi_6^2$-random variable $X''$ with three and six degrees of freedom, respectively.
We denote the $\alpha$-quantile of the variable $Z_1/\sigma^2$ by $q_\alpha$, $\alpha \in (0,1)$.
Then a test for (\ref{e.H_0}) at level $\alpha$ is given by:
\begin{equation*}
 \text{ Reject $H_0$,  if } \quad 
 \frac{n}{\sigma^2} \Gwhat (\hat{\P}_n, \P) \geq q_{1-\alpha} .
\end{equation*}

\medskip

We analyse the protein ubiquitin (consisting of 76 amino acids, PDB reference 1ubq) using the crystallography data (implying $\P$) and the NMR data (with sample size $n=10$) from the Protein Database (RCSB PDB), see~\cite{Berman01012000}. For each of those amino acids we test the hypothesis $H_0$ in \eqref{e.H_0}.
\begin{figure}
\centering
\includegraphics[width = 9cm]{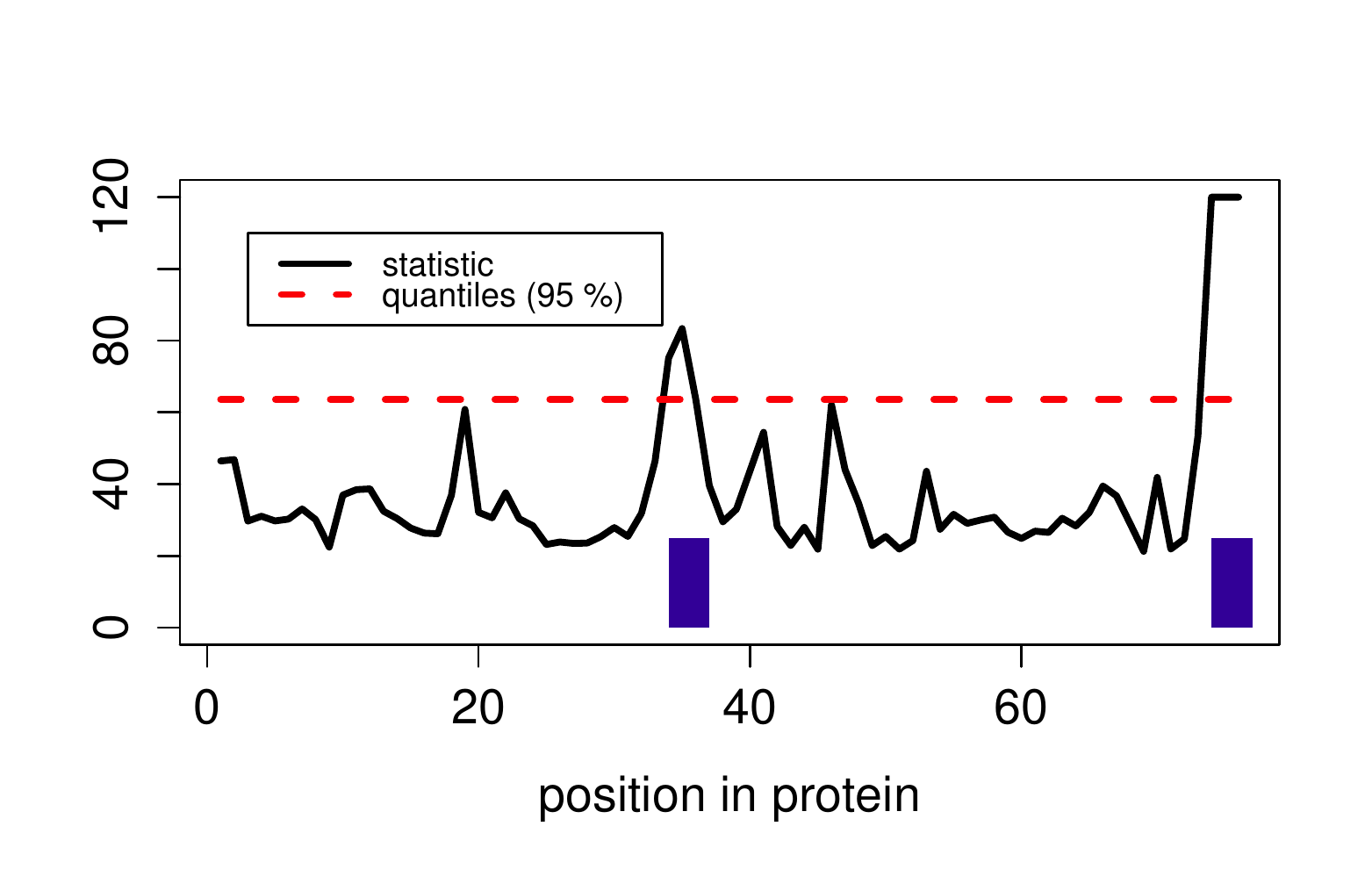}
\caption{Test for $H_0$; regions of rejection in blue.}\label{f.1}
\end{figure}
At level $\alpha = 0.05$ for 8 of the 76 amino acids we reject $H_0$ (see Figure~\ref{f.1})
and at level $\alpha = 0.01$ for 4 of the 76 amino acids we reject $H_0$.
Interestingly, all the rejection appears in the loops of the protein which suggests that NMR and crystallographic structure determination does not align well at these locations. At other locations of the  ubiquitin protein we did not find evidence for deviation from the normal distribution as predicted by the model.

We stress that our analysis does not provide evidence for the positions not being jointly multivariate normal. This is an issue which would require larger samples but could be investigated with our methodology as well.

\section{Proof of Theorems \ref{T.WS.ASYMP.ONESAMPLE}, \ref{T.WS.ASYMP}  and \ref{T.WS.ASYMP:P=Q}}\label{s.2}

In order to prove Theorem \ref{T.WS.ASYMP} and Theorem \ref{T.WS.ASYMP:P=Q} we will apply the Delta method.
To prepare for this, in Section \ref{Sec:Frechetdiff.Proof}, 
we provide the proofs for the Fr\'echet differentiation of the Gaussian Wasserstein estimator from Section \ref{Sec:Frechetdiff}.
Section \ref{Sec:Deltamethod} collects the required standard results on the convergence of the empirical mean and covariance matrix of Gaussian distributions.
Combining these results with the differentiation results we complete the proof of Theorems \ref{T.WS.ASYMP.ONESAMPLE} and \ref{T.WS.ASYMP} with the exception of determining the variance of the limit which will be provided in Section \ref{Sec:variance}.
The proof of Theorem  \ref{T.WS.ASYMP:P=Q} follows similar arguments and is also completed in Section \ref{Sec:Deltamethod}.

\subsection{\Frechet{} differentiability}
\label{Sec:Frechetdiff.Proof}

First we sketch an application of the results from Section \ref{s.2nd.order}.
Let 
\begin{equation*}\DD = L(\Rdim,\Rdim) \simeq \R^{d \times d}\end{equation*}
be the set of (continuous) linear maps from $\Rdim$ to  $\Rdim$ and let $\GG=\IC$ be the complex numbers.
Clearly, $\DD$ is a Banach algebra with respect to the classical operator norm $\|A\|=\| A\|_\DD = \sup_{||x||=1} ||Ax||, A \in \DD.$ Consider the subspace $S_+(\Rdim)\subset \DD$ of symmetric, positive definite matrices (which means that all eigenvalues are positive).
Then any $A \in S_{+}(\Rdim)$ can be written in the form
\begin{equation}
\label{hermitdecomp}
A= \sum_{i=1}^{d} \lambda_i P_i ,
\end{equation}
$\lambda_i \in (0,\infty)$, $P_i P_j = \delta_{ij}P_i$ for $1\leq i \leq j \leq d.$ 
Note that this definition is different from \cite{GHJR:09} in that repeated eigenvalues are listed according to their multiplicity.
Suppose $\psi: D \to \IC$ is analytic.
Then it is possible to define $\psi(A')$ for $A'$ in a neighborhood of $A \in S_{+}(\Rdim)$ as in \eqref{e.tr4} and apply Corollary \ref{cor:Second derivative}.
We will do that in the two proofs to come below.

The first proof deals with the derivative of the Gaussian Wasserstein distance functional.

\begin{proof}[Proof of Proposition \ref{p:Phi:diff}]
 The mapping $\phi^{(1)}: \R^{2d} \to \R, \ (\mu,\nu) \mapsto \|\mu -\nu\|^2$ has \Frechet{}-derivative
 \begin{align}
 D_{(\mu,\nu)}\phi^{(1)} [(g,g')]  = 2 \langle \mu-\nu, g-g' \rangle.
  \label{eq:pf:mean} 
    \end{align}
Next, we treat the second part of the mapping by considering $\phi^{(2)}: S_{+}(\Rdim)^2\to \R$ given by 
\[
(A,B) \mapsto  \phi^{(2)}(A,B)=\tr(A) + \tr(B) - 2 \tr\left( (A^{1/2}BA^{1/2})^{1/2} \right).
\]
Here, we first consider $\psi: D \to \IC, z \mapsto z^{1/2}$, where $D$ is some open bounded subset of $\IC$ not containing elements of the ray $\R_{-}=(-\infty, 0].$ An application of Corollary \ref{cor:Second derivative} yields its Fr\'echet derivative at $A \in  S_{+}(\Rdim)^2,$
\begin{equation}\label{eq:ghjr:root} 
D_A\psi[A']  = \sum_{i=1}^d \frac{1}{2\lambda_i^{1/2}} P_i A'P_i + \sum_{\substack{i,k=1 \\ \lambda_i \neq \lambda_k}}^d \frac{\lambda_k^{1/2} - \lambda_i^{1/2}}{\lambda_k - \lambda_i} P_i A' P_k  
\end{equation}
for any $A' \in \R^{d \times d}.$ Using this we can deduce the \Frechet{} derivative of
\[ 
\phi:S_+(\R^\dimension)^2 \to Y,\ (A,B) \mapsto A+B - 2 \left( A^{1/2}BA^{1/2}\right)^{1/2}. 
\]
First note that by linearity, Lemmas \ref{lem:FrechetProductRule}, \ref{lem:FrechetChainrule}, and \ref{lem:FrechetProjection},
\begin{align*}
 D_{(A,B)} & \phi[(G,G')] = G + G' \\
 &-2 D\psi_{A^{1/2}BA^{1/2}} \left[ D\psi_{A} [G]  BA^{1/2} + A^{1/2} G' A^{1/2} + A^{1/2} B D\psi_{A} [G] \right] ,
\end{align*}
$G, G' \in \R^{d \times d}$. With \eqref{eq:lambda,kappa} and \eqref{eq:ghjr:root} we can write the derivative more explicitly
\begin{align*}
 & D_{(A,B)}  \phi[(G,G')]  = G+G' \\
 &- 2 \sum_{l=1}^\dimension \frac{1}{2\kappa_l^{1/2}}  \left(\sum_{i=1}^\dimension \frac{1}{2\lambda_i^{1/2}} Q_l P_i G P_i B \sqrt{A} Q_l  
 +  Q_l\sqrt{A} G' \sqrt{A} Q_l \right.\\
 &\phantom{AAAAAAAAAAAAAAAAAAAA}+ \left.\sum_{j=1}^\dimension \frac{1}{2\lambda_j^{1/2}} Q_l \sqrt{A} B P_j G P_j Q_l \right) \\
 &- 2 \sum_{\substack{l,k=1 \\ \kappa_l \neq \kappa_k}}^d \frac{\kappa_l^{1/2}- \kappa_k^{1/2}}{\kappa_l-\kappa_k}  \left( \sum_{i=1}^d  \frac{1}{2\lambda_i^{1/2}} Q_l P_i G P_i B  \sqrt{A} Q_k + Q_l \sqrt{A} G' \sqrt{A} Q_k \right.\\
 &\phantom{AAAAAAAAAAAAAAAAAAAAAAA}+ \left.   \sum_{j=1}^d \frac{1}{2\lambda_j^{1/2}} Q_l \sqrt{A} B P_j G P_j Q_k \right) \\
 &- 2 \sum_{l=1}^\dimension \frac{1}{2\kappa_l^{1/2}} \sum_{\substack{i,m=1 \\ \lambda_i \neq \lambda_m}}^d \frac{\lambda_i^{1/2}-\lambda_m^{1/2}}{\lambda_i - \lambda_m} \left( Q_l P_i G P_m B \sqrt{A} Q_l + Q_l \sqrt{A} B P_i G P_m Q_l \right) \\
 &- 2 \sum_{\substack{l,k=1 \\ \kappa_l \neq \kappa_k}}^d  \frac{\kappa_l^{1/2}- \kappa_k^{1/2}}{\kappa_l-\kappa_k}  \sum_{\substack{i,m=1 \\ \lambda_i \neq \lambda_m}}^d \frac{\lambda_i^{1/2}-\lambda_m^{1/2}}{\lambda_i - \lambda_m} \\
 & \phantom{AAAAAAAAAAAAAA} \left( Q_l P_i G P_m B \sqrt{A} Q_k + Q_l \sqrt{A} B P_i G P_m Q_k \right).
\end{align*}
Using the fact that 
\begin{align*}
\lambda_i^{1/2}P_iB\sqrt{A}Q_l = P_i\sqrt{A} B\sqrt{A} Q_l = \kappa_l P_i Q_l ,\\
\lambda_j^{1/2}Q_l\sqrt{A}BP_j = Q_l \sqrt{A} B \sqrt{A} P_j = \kappa_l Q_l P_j,
\end{align*}
allows us to simplify the above expression to
\begin{align*}
 D_{(A,B)} & \phi[(G,G')]  = G+G' \\
 - & \sum_{l=1}^\dimension \kappa_l^{1/2} \Big( \sum_{i=1}^\dimension  \frac{1}{2\lambda_i} Q_l P_i G P_i Q_l  + \frac{1}{\kappa_l} Q_l\sqrt{A} G' \sqrt{A} Q_l \\
 & \phantom{AAAAAAAAAAAAAAAAAA} + \sum_{j=1}^\dimension \frac{1}{2\lambda_j} Q_l P_j G P_j Q_l \Big) \\
 - 2& \sum_{\substack{l,k=1 \\ \kappa_l \neq \kappa_k}}^d \frac{\kappa_l^{1/2}- \kappa_k^{1/2}}{\kappa_l-\kappa_k}  \Big( \sum_{i=1}^d  \frac{\kappa_k}{2\lambda_i} Q_l P_i G P_i  Q_k +  Q_l \sqrt{A} G' \sqrt{A} Q_k  \\
 & \phantom{AAAAAAAAAAAAAAAAAA}+  \sum_{j=1}^d \frac{\kappa_l}{2\lambda_j} Q_l  P_j G P_j Q_k \Big) 
 \end{align*}
 \begin{align*}
 - & \sum_{l=1}^\dimension \kappa_l^{1/2} \sum_{\substack{i,m=1 \\ \lambda_i \neq \lambda_m}}^d \frac{\lambda_i^{1/2}-\lambda_m^{1/2}}{\lambda_i - \lambda_m} \left(\frac{1}{\lambda_m^{1/2}} Q_l P_i G P_m  Q_l + \frac{1}{\lambda_i^{1/2}} Q_l  P_i G P_m Q_l \right) \\
 - 2 & \sum_{\substack{l,k=1 \\ \kappa_l \neq \kappa_k}}^d  \frac{\kappa_l^{1/2}- \kappa_k^{1/2}}{\kappa_l-\kappa_k}  \sum_{\substack{i,m=1 \\ \lambda_i \neq \lambda_m}}^d \frac{\lambda_i^{1/2}-\lambda_m^{1/2}}{\lambda_i - \lambda_m}  \Big(\frac{\kappa_k}{\lambda_m^{1/2}} Q_l P_i G P_m  Q_k  \\
 & \phantom{AAAAAAAAAAAAAAAAAAAAAA}+ \frac{\kappa_l}{\lambda_i^{1/2}} Q_l  P_i G P_m Q_k \Big).
\end{align*}
Note that this can be simplified further as several of the terms are now of the same form.
However, in the end we will take the trace of this object which will lead to further reductions.
We will perform these steps at the same time.
The trace is a linear mapping so that with Lemma \ref{lem:FrechetChainrule} we  obtain
\[ \nn D_{(A,B)}  \left( \tr \circ \phi \right) \left[(G, G') \right]  = \tr  \left(D  \phi_{(A,B)}(G,G') \right). \]
Now use  that $\tr (A) = \sum_{i=1}^\dimension q_i^t A q_i$ for any operator $A$ (see Lemma \ref{lem:trace-formula}) with the eigenbasis $\{q_l: l =1,\dots, \dimension\}$ where $Q_l =q_l q_l^t$.
Then all of the terms containing $Q_l$ and $Q_k$ for $l\neq k$  vanish leaving us with
\begin{equation}
  D_{(A,B)}  \phi^{(2)} [(G, G')]  =D  \left( \tr \circ \phi \right)_{(A,B)} [(G, G')]  \label{eq:DPhi}  
\end{equation}
\begin{align*}
 & \quad = \tr G + \tr G'  -  \sum_{l=1}^\dimension \kappa_l^{1/2} \sum_{i=1}^\dimension  \lambda_i^{-1} q_l^t P_i G P_i q_l   - \sum_{l=1}^\dimension \kappa_l^{-1/2} q_l^t \sqrt{A} G' \sqrt{A} q_l  \\
 &  \qquad \qquad \qquad \qquad -  \sum_{l=1}^\dimension \kappa_l^{1/2} \sum_{\substack{ i, m=1\\ \lambda_i \neq \lambda_m}}^\dimension \left( \lambda_i  \lambda_m\right)^{-1/2}  q_l^t P_i G P_m  q_l.
\end{align*}
Adding this to (\ref{eq:pf:mean}) ends the proof due to Lemma \ref{lem:FrechetProjection}.
\end{proof}
%

%

In a next step we give the proof for the result on the second order differentiability.
\begin{proof}[Proof of Theorem \ref{p:Phi:diff:2}]
 We need to check that the first derivative $D\Phi$ obtained in Proposition \ref{p:Phi:diff} is Fr\'echet differentiable.
 Formally, by chain rule and linearity of the trace,
 \begin{align}
 \nonumber  D^2_{(\mu,\nu,A,B)}\Phi \left[ (g,g',G,G'),(g,g',G,G') \right] & = 2 \langle g- g', g-g' \rangle  \\
 \label{e.tr1}  &\ + \tr \left(D^2_{(A,B)} \Psi \left[(G,G'),(G,G') \right] \right), 
 \end{align}
where $\Psi(A,B) = (A^{1/2}BA^{1/2})^{1/2} = \psi (\psi(A)B\psi(A))$ with $\psi(C) = C^{1/2}$.
This formal derivation is valid as long as the last expression $D^2_{(A,B)} \Psi[(G,G'),(G,G')]$ exists.

First, let us note that $\psi: S_+(\Rdim) \to S_+(\Rdim)$ is twice Fr\'echet differentiable by Corollary \ref{cor:Second derivative}.  
Then the existence can be obtained from the chain rule in Lemma \ref{lem:FrechetChainrule}, more precisely:
\begin{align*}
 D^2_{(A,B)} & \Psi[(G,G'), (G,G') ]  = D^2_{A^{1/2}BA^{1/2}} \psi \left[ C,C \right] \\
  & + D_{A^{1/2}BA^{1/2}}\psi \left[D^2_{(A,B)}(\psi(A)B\psi(A))[(G,G'),(G,G')] \right] \, ,
\end{align*}
where $C = D_{(A,B)}(\psi(A)B\psi(A))[(G,G')]$ is used for abbreviation.
The objects in the first line are all well-defined since $\psi$ is twice Fr\'echet differentiable and 
$\psi(A)B\psi(A)=A^{1/2}BA^{1/2}$. Note that by Lemma \ref{lem:FrechetProductRule} 
the objects in the second line are also well-defined.
\begin{align*}
 D^2_{(A,B)} & (\psi(A)B\psi(A))[(G,G'),(G,G')] \\
 & = D^2_A \psi [G,G] B \psi(A) + 0 + \psi(A) B D^2_A \psi[G,G] \\
 & \quad + 2 D_A\psi[G]G' \psi(A)   + 2 D_A\psi[G]B D_A \psi[G]  +2 \psi(A) G' D_A\psi[G]  \, .
\end{align*}
This means that we have defined all elements in \eqref{e.tr1} rigorously, hence $\Phi$ is twice Fr\'echet differentiable.
\end{proof}
In the case $d=1$ (so $A,B$ are real-valued) we can explicitly calculate the second derivative:
\begin{equation}\label{e.tr3}\begin{split}
D^2 & \Phi_{(\mu,\nu,A,B)} [(g,g',G,G'),(g,g',G,G')] \\
&\quad =  2 (g- g')^2 + \frac{1}{2A^{1/2}B^{1/2}} \Big( \frac{B}{A}G^2 + \frac{A}{B}(G')^2 - 2GG' \Big) \, .
\end{split}
\end{equation}

\subsection{The Delta method and proof of Theorems \ref{T.WS.ASYMP} and \ref{T.WS.ASYMP:P=Q}}
\label{Sec:Deltamethod}
The goal of this section is to derive Theorems \ref{T.WS.ASYMP} and \ref{T.WS.ASYMP:P=Q} via the Delta method. 
More precisely, we will use the following result.

\begin{theorem}[Theorem 20.8 of \cite{vdV:98}, Delta Method]\label{thm:vdv:delta}
 Let $\phi: \bar{\DD} \subset \DD \to \GG$ be \Frechet{} differentiable at $\theta \in \bar{\DD}.$
 Let $(T_n)_{n \in \IN}$ and $T$ be random variables with values in  $\bar{\DD}$ and $\DD$ respectively  such that $r_n(T_n -\theta) \Rightarrow T$ for some sequence of numbers $r_n \to \infty.$ 
 Then 
 $$r_n(\phi(T_n)-\phi(\theta)) \Rightarrow D_\theta\phi[T].$$ 
If additionally, $D_\theta \phi = 0$ and $\phi$ is twice differentiable at $\theta$, then
$$ r_n^2(\phi(T_n)-\phi(\theta)) \Rightarrow \frac{1}{2} D_\theta^2\phi[T,T].$$ 
 \end{theorem}
\begin{remark}
In \cite{vdV:98} the result is stated in more generality, in particular for Hadamard differentiable functions. Since we essentially work in finite dimensions this difference does not matter. The statement on second derivatives is not included in Theorem 20.8 of \cite{vdV:98}. However, the proof is quite the same using an expansion to a higher order, see Section 20.1.1 of \cite{vdV:98} as well as Theorem \ref{thm:ghjr} of Appendix \ref{s.2nd.order}.
\end{remark}

In order to apply this result we will use known weak convergence results of the empirical means and covariance matrices of Gaussian distributions to their true means and covariance matrices. The  representation in \eqref{eq:was:dist:Gauss}, whose \Frechet{} derivative was calculated in Proposition \ref{p:Phi:diff} (see also Corollary \ref{p:Phi:diff-onesample}), then provides the mapping from mean and covariance matrices  to the 2-Wasserstein distance \eqref{p-Wasserstein} of Gaussian distributions.

For this we now return to the setting of Theorem \ref{T.WS.ASYMP} such that $\IP$ and  $\IQ$ in $\mathcal{M}_1(\IR^d)$ are Gaussian distributions on $\R^{\dimension}$ and $X_i\sim \IP$ are i.i.d.~and independent from $Y_i\sim \IQ$ i.i.d.~for $i \in \IN.$ A central limit theorem for the respective sample means and covariance matrices of a sample of size $n,$
\begin{align}
\hat{\mu}_n &= \frac{1}{n} \sum_{i=1}^n X_i, \qquad \hat{\Sigma}_n = \frac{1}{n-1} \sum_{i=1}^n (X_i-\hat{\mu}_n)(X_i-\hat{\mu}_n)^t, \\
\hat{\nu}_n &= \frac{1}{n} \sum_{i=1}^n Y_i, \qquad \hat{\Xi}_n = \frac{1}{n-1} \sum_{i=1}^n (Y_i-\hat{\nu}_n)(Y_i-\hat{\nu}_n)^t
\end{align}
 is well known.
\begin{lemma}[Section 3 in \cite{RY:97}]\label{lem:anderson}
  If $\IP = N(\mu,\Sigma)$ then
  \begin{equation}
   \sqrt{n} ( \hat{\mu}_n - \mu, \, \hat{\Sigma}_n - \Sigma ) \Rightarrow g \otimes G
  \end{equation}
where $g \sim N(0, \Sigma)$ and $G = \Sigma^{1/2}H\Sigma^{1/2}.$ Here, convergence in the space $\Rdim \times \R^{\dimension \times \dimension}$ is understood component wise and $H= (H_{ij})_{i,j\leq \dimension}$ is a $\dimension \times \dimension$ symmetric random matrix with independent (upper triangular) entries and
\begin{equation}
 H_{ij} \sim \begin{cases}
            N(0,1) \ & ,\, i < j ,\\
            N(0,2) \ & ,\, i = j .
          \end{cases}
\end{equation}
\end{lemma}
The derivation in \cite{RY:97} is only given for the centered case, but (as they also say) it can easily be obtained in the non-centered case.

\bigskip


\smallskip

\noindent
{\bf Main lines of the proof of Theorem \ref{T.WS.ASYMP}:}\\
First, note that due to $\Sigma$ and $\Xi$ having full rank, all eigenvalues are positive.
Consider $\Phi$ as in Proposition \ref{p:Phi:diff}, $\mathbb{D} = \R^{2d} \times \R^{2(d \times d)}, \, \GG = \R$ and $\bar{\DD} = \R^{2d} \times S_{+}(\Rdim)^2.$
For $r_n = \sqrt{mn/(m+n)}$ and $T_n (x_1,\dotsc, x_n, y_1, \dotsc, y_m) = (\hat{\mu}_n,\hat{\nu}_m,\hat{\Sigma}_n,\hat{\Xi}_m)$ we obtain with the help of Lemma \ref{lem:anderson},
\begin{equation}\label{eq:CLT} \begin{split} r_n  & \left(T_n (X_1,\dotsc, X_n,Y_1,\dotsc Y_m) - (\mu,\nu,\Sigma, \Xi) \right) \\
 & \Rightarrow \left( (1-a)^{1/2}g,a^{1/2}g',(1-a)^{1/2}G,a^{1/2}G' \right) \text{ as } n \to \infty 
 \end{split}\end{equation}
with $g \sim N(0,\Sigma),\, g' \sim N(0,\Xi)$ and $G = \sqrt{\Sigma}H\sqrt{\Sigma},\, G' = \sqrt{\Xi}H'\sqrt{\Xi}$ all independent of each other.
The symmetric Gaussian matrices $H$,$H'$ have independent Gaussian entries in the upper triangle with mean $0$ and variance $1$ off-diagonal and variance $2$ on the diagonal.
We can now apply Theorem~\ref{thm:vdv:delta} in order to obtain
 \begin{equation}
 \label{Delta_to_Phi} \begin{split}
  r_n & \left( \Phi(\hat{\mu}_n,\hat{\nu}_n,\hat{\Sigma}_n,\hat{\Xi}_n) - \wass_2^2(\IP,\QQ) \right) \\
   & \Rightarrow D_{(\mu,\nu,\Sigma,\Xi)}\Phi \left[ ((1-a)^{1/2}g,a^{1/2}g',(1-a)^{1/2}G,a^{1/2}G') \right] \, .
 \end{split} \end{equation}
Since $(g,g',G,G')$ is a Gaussian vector with mean $0$ and $D\Phi$ is a linear mapping to $\R$ we know that $D_{(\mu,\nu,\Sigma,\Xi)}\Phi (((1-a)^{1/2}g,a^{1/2}g',(1-a)^{1/2}G,a^{1/2}G') )$ is a real-valued Gaussian variable with mean $0$ and a certain variance $\varpi$.
This shows (\ref{eq:twosample.m.n}).
The calculation of $\varpi$ leading to (\ref{eq:vartwosample.m.n}) is provided in Section~\ref{Sec:variance}.\hfill \qed \\

In the one sample case (\ref{eq:onesample}), i.e.~Theorem~\ref{T.WS.ASYMP.ONESAMPLE} the proof is entirely analogous but essentially simpler: We use 
 Theorem~\ref{thm:vdv:delta} and Lemma~\ref{lem:anderson} as before in order to obtain
 \begin{align}
 \label{Delta_to_Phi^}
  \sqrt{n} ( \Phi(\hat{\mu}_n,\nu,\hat{\Sigma}_n,\Xi) - \wass_2^2(\IP,\QQ) ) \Rightarrow D_{(\mu,\Sigma)}\Phi^{(\nu,\Xi)} [(g,G)] \, 
 \end{align}
 where the derivative is specified in Corollary \ref{p:Phi:diff-onesample}. Again, the limit is mean $0$ Gaussian and the calculation of the variance in (\ref{eq:varonesample}) is given at the end of Section \ref{Sec:variance}.

\medskip 
\begin{remark}\label{r.tr1}
A final remark is to say something about the case when $\Sigma$ or $\Xi$ do not have full rank in Theorem \ref{T.WS.ASYMP}.
Simulations show that still a very similar result should hold.
However, our technique (delta method, i.e.~differentiation) will not work, as can already be seen in the case $d=1$. Loosely speaking the derivative of the variance part in \eqref{eq:was:dist:Gauss} where $\lambda \geq 0$ is an eigenvalue is $\approx \lambda^{-1/2}$ (not being well-defined for $\lambda =0$) which gets multiplied by the direction $\approx \lambda$ (see \eqref{Delta_to_Phi}) yielding $\approx \lambda^{1/2}$ in the end.
\end{remark}

\bigskip

{\bf Proof of Theorem \ref{T.WS.ASYMP:P=Q}:}\\
Let $\Phi$ as in Proposition \ref{p:Phi:diff}.
Note that $\Phi((\mu,\mu,A,A)) = 0$ and the  proposition easily implies that $D_{(\mu,\mu,A,A)} \Phi = 0$. 
Additionally, Proposition~\ref{p:Phi:diff:2} says that the function $\Phi$ is twice \Frechet{} differentiable at the point $(\mu,\mu,A,A)$ and thus we can apply the second part of Theorem \ref{thm:vdv:delta}.
This allows to deduce that
\begin{equation}\label{e.tr2} \begin{split}
n & \left(\Phi(\hat{\mu}_n^{(1)},\hat{\mu}_n^{(2)},\hat{\Sigma}_n^{(1)},\hat{\Sigma}_n^{(2)}) - 0 \right) \\
& \qquad \Rightarrow  D_{(\mu,\mu,\Sigma,\Sigma)}^2 \Phi \left[(g,g',G,G'),(g,g',G,G') \right] \, ,
\end{split} \end{equation}
where $g \sim N(0,\Sigma), G = \sqrt{\Sigma}H\sqrt{\Sigma}$, $g' \sim N(0,\Sigma), G' = \sqrt{\Sigma}H'\sqrt{\Sigma}$ are all independent of each other and as in Lemma \ref{lem:anderson}.
Since $D^2 \Phi$ is a quadratic form and the vector $(g,g',G,G')$ is Gaussian we obtain the desired result. \hfill \qed

\subsection{Variance formula for the limiting Gaussian distributions}\label{sec:VarianceCalc}
\label{Sec:variance}
In this section we provide the details of calculating the variance of the derivative $D_{(\mu,\nu,\Sigma,\Xi)}\Phi (((1-a)^{1/2}g,a^{1/2}g',(1-a)^{1/2}G,a^{1/2}G'))$ in (\ref{Delta_to_Phi}) whose explicit form is given in (\ref{eq:D}) of Proposition \ref{p:Phi:diff}. 
The variance formula for  $D_{(\mu,\Sigma)}\Phi^{(\nu,\Xi)} ((g,G))$ of (\ref{Delta_to_Phi^}) specified in Corollary \ref{p:Phi:diff-onesample} then follows in a similar way with the the calculation in \eqref{e.upsilon} below.

The first two terms of the representation (\ref{eq:D}) involving the means $\mu$ and $\nu$ are easily calculated, namely 
\begin{equation}
\label{meansum}
 \begin{split} & (\mu-\nu)(1-a)^{1/2}g \sim N\left( 0, (1-a)(\mu-\nu)^t\Sigma (\mu-\nu) \right), \\
  & (\mu-\nu)a^{1/2}g' \sim N\left( 0, a(\mu-\nu)^t\Xi (\mu-\nu) \right) \, .
 \end{split}
\end{equation}

The explicit calculation of the remaining terms involving the covariance matrices $\Sigma$ and $\Xi$  is more complicated.
We will frequently apply Lemma \ref{lem:EH}.
%
In the following use the eigendecomposition of $A$ and  $A^{1/2}BA^{1/2}$ given in \eqref{eq:lambda,kappa}.
Let $G = A^{1/2}HA^{1/2}$, where $H$ is as in Lemma \ref{lem:anderson}.
Then since $AP_i = P_iA=\lambda_i P_i$, $1\leq i \leq d$ and $\sum_{i=1}^d P_i = \mathbb{I}$ the terms in (\ref{eq:D}) that involve $G$ are given by
\begin{align}
\nonumber
 & \tr (G) - \sum_{l=1}^d \kappa_l^{1/2} \sum_{i=1}^d \lambda_i^{-1} q_l^t P_i G P_i q_l  \\
 \nonumber    & \phantom{AAAAAAAAA} - \sum_{l=1}^d  \kappa_l^{1/2} \sum_{i,m=1, \lambda_i \neq \lambda_m}^d (\lambda_i \lambda_m)^{-1/2} q_l^t P_i G P_m q_l \\
 \nonumber
 & \ = \tr(AH) - \sum_{l=1}^d \kappa_l^{1/2} \sum_{i=1}^d q_l^t P_i H P_i q_l - \sum_{l=1}^d \kappa_l^{1/2}q_l^t\sum_{i=1}^d P_i H \sum_{m=1, \lambda_i \neq \lambda_m}^d P_m q_l \\
 \nonumber
 & \ = \tr(AH) - \sum_{l=1}^d \kappa_l^{1/2}q_l^t\sum_{i=1}^d P_i H \left( P_i + \sum_{m=1, \lambda_i \neq \lambda_m}^d P_m \right) q_l \\
 & \label{eq:Gpart}
  \ = \tr(AH) - \sum_{l=1}^d \kappa_l^{1/2}q_l^t\sum_{i=1}^d P_i H \tilde{P}_i q_l \, ,
\end{align}
where in the last line we have used the notation
\begin{equation}
\label{tildeP}
 \tilde{P}_i =  P_i + \sum_{m=1, \lambda_i \neq \lambda_m}^d P_m 
\end{equation}
to denote the projection onto the direction corresponding to $\lambda_i$ as well as on all other directions that have eigenvalues different from $\lambda_i.$
For future use we note that  $\tilde{P}_i$ is again a projection due to the orthogonality of the $P_i, i=1, \dots, d$, meaning that
\begin{align}
\tilde{P}_i \tilde{P}_j = \delta_{ij}\tilde{P}_i \, .
\end{align}
Furthermore $\tilde{P}_i$ is symmetric.
With this we can calculate the second moment and thus the variance  of the centered Gaussian of (\ref{eq:Gpart}).
\begin{eqnarray}
\nonumber
 & & \IE\left[ \left(\tr(AH) - \sum_{l=1}^d \kappa_l^{1/2}q_l^t\sum_{i=1}^d P_i H \tilde{P}_i q_l \right)^2 \right] \\
 \label{G1}
 &&\phantom{AAAAAAAAAAAA}= \sum_{i,j=1}^d \IE \left[ p_i^t AHp_ip_j^t AH p_j \right] \\
  \label{G2}
& &  \phantom{AAAAAAAAAAAAA}+ \sum_{l,k=1}^d \kappa_l^{1/2}\kappa_k^{1/2} \sum_{i,j=1}^d \IE \left[q_l^tP_i H \tilde{P}_i q_l q_k^t P_j  H \tilde{P}_i q_k\right] \\
 \label{G3}
 & & \phantom{AAAAAAAAAAAAA} -  2\sum_{j,l=1}^d \kappa_l^{1/2} \sum_{i=1}^d \IE \left[q_i^t AH q_i q_l^t P_j H \tilde{P}_j q_l  \right] \, .
 \end{eqnarray}
 We consider these three terms separately and start with (\ref{G1}). Using Lemma \ref{lem:EH} and $p_i^tp_j = \delta_{ij}$ the first line \eqref{G1} simplifies to
 \begin{align*}
  &\sum_{i,j=1}^d \lambda_i \lambda_j \left[p_i^t(p_ip_j^t)^t p_j+ (p_i^tp_j) \tr(p_i p_j^t) \right] 
  = \sum_{i=1}^d \lambda_i^2 \left[1 +  \tr(p_i p_i^t) \right] = 2 \sum_{i=1}^d \lambda_i^2.
 \end{align*}
 Also with Lemma \ref{lem:EH} we obtain for (\ref{G2})
  \begin{align*}
 &\sum_{k,l=1}^d \kappa_l^{1/2} \kappa_k^{1/2} \sum_{i,j=1}^d \Bigg[ q_l^tP_iP_j q_k q_l^t\tilde{P}_i \tilde{P}_j q_k + (q_l^tP_i \delta_{ij} q_k + q_l^tP_i \1_{\lambda_i \neq \lambda_j} q_k ) \\
 & \phantom{aaaaaaaaaaaaaaaaaaaaaaaa} \bigg( p_i^tq_lq_k^tp_i \delta_{ij}  + \sum_{m=1,\lambda_i \neq \lambda_m}^d p_m^t q_l q_k^t p_m \1_{j=m} \bigg) \Bigg] \\
&  \stackrel{ q_l^tq_k =\delta_{kl}}{=} \sum_{k,l=1}^d \kappa_l^{1/2} \kappa_k^{1/2} \sum_{i=1}^d q_l^tP_iq_k q_l^t\tilde{P}_i q_k  + \sum_{k,l=1}^d \kappa_l^{1/2} \kappa_k^{1/2} \sum_{i=1}^d q_l^t P_i q_k p_i^t q_l q_k^t p_i 
\end{align*}
\begin{align*}
  & \quad \quad+ \sum_{m=1,\lambda_i\neq \lambda_m}^d q_l^t P_i q_k p_j^t q_l q_k^t p_m \1_{j=m,i=j}  \\
  & \quad \quad
   + \sum_{k,l=1}^d \kappa_l^{1/2} \kappa_k^{1/2} \sum_{i=1}^d \sum_{j=1,\lambda_i\neq \lambda_j}^d q_l^tP_i q_k p_i^tq_l q_k^t p_i + q_l^t P_i q_k p_j^t q_l q_k p_j \\
   &= 2 \sum_{k,l=1}^d \kappa_l^{1/2} \kappa_k^{1/2} \left( \sum_{i=1}^d  (q_l^tP_iq_k)^2 + \sum_{j=1,\lambda_i\neq \lambda_j}^d q_l^tP_i q_k q_l^tP_jq_k \right)  \, .
 \end{align*}
 Similarly, for (\ref{G3}) we get
 \begin{align*}
 & \quad  - 2\sum_{j,l=1}^d \kappa_l^{1/2} \sum_{i=1}^d \left[ q_i^t A P_j q_l q_i^t \tilde{P}_j q_l + q_i^tA\tilde{P}_i q_l \sum_{r=1}^d q_r^tq_iq_l^t P_j q_r  \right] \\
  & \quad  - 2 \sum_{j,l=1}^d \kappa_l^{1/2}  \lambda_j q_l^tP_j P_j q_l + \lambda_j q_l^t P_j  \sum_{m=1,\lambda_j \neq \lambda_m}^d P_m q_l \quad (\text{since} \sum_{i=1}^d q_i q_i^t = 1) \\
 & \quad - 2 \sum_{j,l=1}^d \kappa_l^{1/2}  \lambda_j q_l^tP_jP_j q_l + \lambda_jq_l^t P_j  \sum_{m=1,\lambda_j \neq \lambda_m}  P_m q_l \\
 &= - 4\sum_{l=1}^d \kappa_l^{1/2}  ( q_l^t A q_l + 0 )\, .
 \end{align*}
 By putting (\ref{G1}) to (\ref{G3}) back together and adding the factor $(1-a)$, since in \eqref{e.tr2} we are dealing with $(1-a)^{1/2}G$ instead of $G$ we finally obtain
\begin{align}
\label{Gsum}
 &\quad  \IE\left[ \left((1-a)^{1/2}\tr(AH) - \sum_{l=1}^d \kappa_l^{1/2}q_l^t\sum_{i=1}^d P_i (1-a)^{1/2}H \tilde{P}_i q_l \right)^2 \right] \\
 \nonumber
 & = 2 (1-a)\tr (A^2) + 2 (1-a) \sum_{k,l=1}^d \kappa_l^{1/2} \kappa_k^{1/2} \sum_{i=1}^d q_l^tP_iq_k q_l^t \tilde{P}_i q_k  \\
 & \phantom{AAAAAAAAAAAAAAAAAAAAAAA}    -4 (1-a) \sum_{l=1}^d \kappa_l^{1/2} q_l^t A q_l \, .
\end{align}
%
%
%
We can do a similar calculation for the variance related to $G' = B^{1/2}H'B^{1/2}:$
\begin{align}
\nonumber
 \IE & \left[ \left( \tr (G') - \sum_{l=1}^d \kappa_l^{-1/2} q_l^t A^{1/2} G' A^{1/2} q_l \right)^2 \right]
\\
 \label{G'1}
  &= \sum_{k,l=1}^d  \Big(\IE \left[ q_l^tBH'q_lq_k^tBH'q_k \right]  
  \end{align}
\begin{align}
   \label{G'2}
  &\qquad +  \IE \left[\kappa_l^{-1/2}\kappa_k^{-1/2} q_l^t A^{1/2} B^{1/2} H' B^{1/2}A^{1/2} q_l q_k^t A^{1/2}B^{1/2}H' B^{1/2}A^{1/2}q_k \right] \\
   \label{G'3}
 &\qquad - 2 \IE \left[ \kappa_l^{-1/2}q_k^t BH'q_kq_l^tA^{1/2}B^{1/2}H'B^{1/2}A^{1/2}q_l \right] \Big).
 \end{align}
Here, the first term in (\ref{G'1}) simplifies with the help of Lemma \ref{lem:EH} and Lemma \ref{lem:ev(AB)} as well as  $q_l^tq_k = \delta_{kl}$ and  $\tr(q_lq_k^tB) =  q_k^tBq_l$ to 
 \begin{align*}
 \sum_{k,l=1}^d \left( q_l^tB(q_l q_k^tB)^tq_k + q_l^tBq_k \tr(q_lq_k^tB)  \right) 
 &= \sum_{l=1}^d q_l^tB^2 q_l + \sum_{k,l=1}^d  q_l^tBq_k q_k^tB q_l  \\
 &= 2 \sum_{l=1}^d q_l^tB^2 q_l  =2 \tr (B^2).
 \end{align*}
 Using Lemma \ref{lem:EH} and the fact that $\kappa_i$ and $q_i$ are the eigenvalues and orthonormal eigenvectors of $A^{1/2}BA^{1/2}$ the second term in  (\ref{G'2}) reduces to
  \begin{align*}
&  \sum_{k,l=1}^d \kappa_l^{-1/2} \kappa_k^{-1/2} \left( q_l^t A^{1/2}B^{1/2} \left(B^{1/2}A^{1/2}q_lq_k^tA^{1/2}B^{1/2} \right)^tB^{1/2}A^{1/2} q_k \right.\\
 & \phantom{AAAAAAAAAAA}  \left. + q_l^tA^{1/2}B^{1/2}B^{1/2}A^{1/2}q_k \tr(B^{1/2}A^{1/2}q_lq_k^tA^{1/2}B^{1/2} ) \right) \\
 =& \sum_{k,l=1}^d  \kappa_l^{-1/2} \kappa_k^{-1/2} \left( q_l^t A^{1/2}BA^{1/2}q_kq_l^tA^{1/2}BA^{1/2} q_k \right. \\
 & \phantom{AAAAAAAAAAA}   \left. + q_l^tA^{1/2}BA^{1/2}q_k \tr \left(q_k^tA^{1/2}B^{1/2}B^{1/2}A^{1/2}q_l \right) \right)  \\
 =&  \sum_{k,l=1}^d  \kappa_l^{1/2} \kappa_k^{1/2} \left( q_l^tq_kq_l^tq_k + q_l^tq_kq_k^tq_l \right) =2\sum_{k=1}^d \kappa_k= 2 \tr(AB). 
  \end{align*}
Finally, with Lemmas \ref{lem:EH} and \ref{lem:trace-formula} the third term in (\ref{G'3}) leads to 
\begin{align*}
& -  2 \sum_{k,l=1}^d\kappa_l^{-1/2} \big( q_k^tB(q_kq_l^tA^{1/2}B^{1/2} )^t B^{1/2}A^{1/2}q_l \\
 &\phantom{AAAAAAAAAAAAAA} + q_k^t B B^{1/2}A^{1/2}q_l \tr(q_kq_l^tA^{1/2}B^{1/2}) \big) \\
=&- 2\sum_{k,l=1}^d \kappa_l^{-1/2} \left( q_k^tB^{1/2}A^{-1/2}A^{1/2}BA^{1/2}q_l q_k^t B^{1/2}A^{1/2}q_l \right. \\
&  \phantom{AAAAAAAAAAAAAA} \left. +q_k^t  B^{1/2}A^{-1/2}A^{1/2}BA^{1/2}q_l q_l^tA^{1/2}B^{1/2}q_k \right) \\
=&  - 2\sum_{k,l=1}^d \kappa_l^{1/2} \left( q_k^tB^{1/2}A^{-1/2}q_l q_k^t B^{1/2}A^{1/2}q_l + q_k^t  B^{1/2}A^{-1/2}q_l q_l^tA^{1/2}B^{1/2}q_k \right) \\
=& - 4 \sum_{l=1}^d \kappa_l^{1/2} q_l^t A^{-1/2}B A^{1/2} q_l .
 \end{align*}
Thus, we obtain from the simplifications of (\ref{G'1}) to (\ref{G'3}) and using the factor $a$ from \eqref{e.tr2},
\begin{align}
\label{G'sum}
 \IE & \left[ \left( a^{1/2}\tr (G') - a^{1/2}\sum_{l=1}^d \kappa_l^{-1/2} q_l^t A^{1/2} G' A^{1/2} q_l \right)^2 \right]  \\
 \nonumber
 &= 2 a  \tr (B^2) +2a \tr(AB) - 4 a\sum_{l=1}^d \kappa_l^{1/2} q_l^t A^{-1/2}B A^{1/2} q_l .
\end{align}
Finally, from (\ref{meansum}), (\ref{Gsum}) and (\ref{G'sum}) (now replacing $A$ and $B$ by $\Sigma$ and $\Xi$) as well as the independence of $g,g', G,G'$ we obtain that the variance  in (\ref{Delta_to_Phi}) of the random variable $D_{(\mu,\nu,\Sigma,\Xi)}\Phi ((g,g',G,G'))$ is given by 
\begin{align*}
&(\mu-\nu)^t ((1-a)\Sigma + a\Xi) (\mu-\nu) +2 (1-a)\tr (\Sigma^2)\\
& + 2 (1-a) \sum_{k,l=1}^d \kappa_l^{1/2} \kappa_k^{1/2} \sum_{i=1}^d q_l^tP_iq_k q_l^t \tilde{P}_i q_k  
   -4 (1-a)\sum_{l=1}^d \kappa_l^{1/2} q_l^t \Sigma q_l\\
&+2 a \tr (\Xi^2) +2 a \tr(\Sigma \Xi) - 4 a \sum_{l=1}^d \kappa_l^{1/2} q_l^t \Sigma^{-1/2}\Xi \Sigma^{1/2} q_l.
\end{align*}
If all eigenvalues are distinct we have that $\tilde{P}_i=I, i = 1,\dots d$ and therefore using $q_l^t q_k=\delta_{kl}$ it also follows that 
\begin{align*}
&\sum_{k,l=1}^d \kappa_l^{1/2} \kappa_k^{1/2} \sum_{i=1}^d q_l^tP_iq_k q_l^t \tilde{P}_i q_k  
= \sum_{l=1}^d \kappa_l \sum_{i=1}^d q_l^tP_iq_l \\
&=\sum_{l=1}^d \kappa_l\sum_{i=1}^d q_l^tp_i p_i^tq_l  =\sum_{l=1}^d \kappa_l\sum_{i=1}^d p_i^tq_l  q_l^tp_i 
= \sum_{l=1}^d \kappa_l \tr{Q_l} = \tr(\Sigma \Xi).
\end{align*}
Thus, in this case the expression for the variance reduces to 
\begin{align*}
&(\mu-\nu)^t ((1-a)\Sigma + a\Xi) (\mu-\nu) +2 \tr ((1-a)\Sigma^2 + a\Xi^2)+2 \tr(\Sigma \Xi)\\
&\qquad \qquad \qquad \qquad \qquad    -4 \sum_{l=1}^d \kappa_l^{1/2} q_l^t( (1-a)\Sigma + a\Sigma^{-1/2}\Xi \Sigma^{1/2}) q_l.
\end{align*}
We have chosen to also use this representation for the general case together with the fact that  by (\ref{tildeP}),
\begin{equation}
 \tilde{P}_i =  I - \sum_{\stackrel{j=1}{j \neq i, \lambda_i = \lambda_j}}^d P_j. 
\end{equation}
This yields (\ref{eq:vartwosample.m.n}) of Theorem \ref{T.WS.ASYMP}.

To obtain (\ref{eq:varonesample}) of Theorem \ref{T.WS.ASYMP.ONESAMPLE} we need to be careful.
Recall that in Corollary \ref{p:Phi:diff-onesample} the derivative is given with the terms of $(\mu,A)$ and $(\nu,B)$ being reversed.
So we need to follow the previous calculation for $(g,G)=0$ and reverse the roles of $(\mu,\Sigma)$ and $(\nu,\Xi)$ finally, i.e.~set $A= \Xi$ and $B=\Sigma$.
So we only obtain the second term in \eqref{meansum} and the terms in \eqref{G'sum} for $a=1$.
\begin{align}\label{e.upsilon}
 \upsilon^2 = (\nu-\mu)^t\Sigma (\nu-\mu) +  2 \tr (\Sigma^2) +2 \tr(\Xi \Sigma) - 4 \sum_{l=1}^d \kappa_l^{1/2} r_l^t \Xi^{-1/2}\Sigma \Xi^{1/2} r_l \, .
\end{align}
Here, $\{(\kappa_l,r_l):\, l=1,\dotsc, d\}$ is the eigendecomposition of $\Xi^{1/2}\Sigma \Xi^{1/2}$.

\appendix
\renewcommand*{\thesection}{\Alph{section}}

\section{Functional derivatives: A reminder}\label{s.Functional.deriv}

We start by collecting some basic facts on \Frechet{} differentiability in an abstract setting. Let 
$\DD$ and $\GG$ be normed linear spaces, $\bar{\DD}\subset \DD$ open and $\phi:\bar{\DD} \to \GG$. The function $\phi$  is \emph{\Frechet{} differentiable} at $\theta \in \bar{\DD}$ if there exists a continuous, linear map $D_{\theta}\phi:\DD \to \GG$ such that as $||h||_{\DD}\rightarrow 0,$
 \[ 
  \frac{1}{||h||_{\DD}}||(\phi(\theta + h)- \phi(\theta)) -D_{\theta}\phi[h]||_{\GG} \rightarrow 0.
  \]
%

This concept also extends to higher order derivatives. E.g.~for the second derivative in the setting above, the mapping $D_{\cdot}:\,\bar{\DD} \to L(\DD,\GG)$ is asked to be \Frechet{} differentiable; here $L(\DD,\GG)$ denotes the space of continuous linear mappings from $\DD \to \GG$. Since the second derivative is a bilinear form it suffices to define it on the diagonal elements.
In the following we collect a number of calculation rules for \Frechet{} derivatives that will be used frequently later on. References for the results are \cite[Section 3.9]{vdVW:96}, Section 3 in \cite{Cheney:01} or the classical sources \cite{D:69} and \cite{AS:67} for a general overview.
First, if $(\GG,\cdot)$ is a Banach algebra then a product rule holds.
\begin{lemma}[Product rule]
\label{lem:FrechetProductRule}
 Suppose that $\phi: \bar{\DD}\subset \DD \to \GG$, $\psi:\bar{\DD}\subset \DD \to \GG$ are \Frechet{} differentiable. 
 Then their product $\phi\cdot \psi : \bar{\DD} \subset \DD \to \GG$ is also \Frechet{} differentiable in $\bar{\DD}$ and 
 $$D_{\theta}(\phi\cdot \psi)[h] = D_{\theta}\phi[h]  \cdot  \psi(\theta) +   \phi(\theta)  \cdot D_{\theta}\psi[h] , \qquad h \in \bar{\DD}.$$
 Additionally, if $\phi: \bar{\DD}\subset \DD \to \GG$, $\psi:\bar{\DD}\subset \DD \to \GG$ are twice \Frechet{}-differentiable, then its product $\phi\cdot \psi : \bar{\DD} \subset \DD \to \GG$ is also twice \Frechet{}-differentiable in $\bar{\DD}$ and for $\theta, h \in \bar{\DD},$
 $$D^2_{\theta}(\phi \cdot \psi)[h,h] 
 = D^2_{\theta}\phi[h,h] \cdot \psi(\theta) +2 D_{\theta}\phi[h] \cdot D_{\theta}\psi[h] +  \phi(\theta) \cdot D^2_{\theta}\psi[h,h].$$
\end{lemma}
\noindent
 We also have a chain rule.
\begin{lemma}[Chain rule]
\label{lem:FrechetChainrule}
 Let $\phi: \bar{\DD} \subset \DD \to \GG$ and $\psi: \bar{\GG} \subset  \GG \to \EE$ with $\phi(\bar{\DD}) \subset \bar{\GG}$ be \Frechet{} differentiable at $\theta \in \bar{\DD}$, $\psi(\theta)\in \bar{\GG}$ respectively. Then $\psi \circ \phi$ is \Frechet{} differentiable at $\theta$ with derivative
 \[ 
 D_{\theta}(\psi\circ \phi)[h] = D_{\phi(\theta)} \psi[ D_{\theta}\phi[h]],\qquad h \in \DD.
 \]
 Here, the right hand side is a linear mapping from $\DD$ to $\EE$. If $\phi$ and $\psi$ are twice \Frechet{} differentiable at the respective points, then $\psi \circ \phi$ is twice \Frechet{} differentiable at $\theta$ with second derivative given by the quadratic form
 \begin{equation}
 \nonumber
  D^2_{\theta} (\psi \circ \phi)[h,h] = D^2_{\phi(\theta)} \psi[ D_{\theta}\phi[h],D_{\theta}\phi[h]]
  +D_{\phi(\theta)} \psi[ D^2_{\theta}\phi[h,h]] , \qquad h \in \DD .
 \end{equation}
\end{lemma}
\noindent
The second part of the lemma 
can be deduced as in the finite-dimensional case.
It is also an elementary observation to obtain the following result on the \Frechet{} derivative of projections.
\begin{lemma}[Projection]
\label{lem:FrechetProjection}
Let $\DD=\DD_1 \times \DD_2$ be a product space of two normed spaces and $\bar{\DD}_1 \subset \DD_1$ open. Let $\phi:\bar{\DD}_1\to \GG$ be 
\Frechet{} differentiable  on $\bar{\DD}_1$ with \Frechet{} derivative $D_{\theta_1}\phi$ at the point $\theta_1 \in \bar{\DD}_1.$ Then $\psi: \bar{\DD}_1 \times \DD_2 \to \GG,$ $(\theta_1, \theta_2) \mapsto \phi(\theta_1)$ is \Frechet{} differentiable in $\bar{\DD}_1 \times \DD_2$ with \Frechet{} derivative
$D_{(\theta_1,\theta_2)}\psi ((h_1,h_2)) = $ $D_{\theta_1}\phi(h_1),$ $h_1\in \DD_1, h_2 \in \DD_2.$
\end{lemma}
\begin{proof}
We have that for $(\theta_1,\theta_2) \in \bar{\DD}_1 \times \DD_2,$
 \begin{align*}
&  \lim_{||(h_1,h_2)||_{\DD} \to 0 } \frac{1}{||(h_1,h_2)||_\DD} ||\psi((\theta_1,\theta_2) + (h_1,h_2)) -\psi((\theta_1,\theta_2) ) - D_{\theta_1}\phi(h_1)||_\GG \\
& \leq  \lim_{||h_1||_{\DD_1} \to 0 } \frac{1}{||h_1||_{\DD_1}} ||\phi(\theta_1+h_1) -\phi(\theta_1) - D_{\theta_1}\phi(h_1)||_\GG =0.
\end{align*}
\end{proof}

\section{A second order result on \Frechet{} derivatives}\label{s.2nd.order}

We closely follow Chapter 3 of \cite{GHJR:09} and extend their results to a derivative of second order. Consider a separable Hilbert space $\mcH$ and the class of bounded linear operators $\mcL$ from $\mcH$ to $\mcH$. Its subclasses of Hermitian and compact Hermitian operators are denoted by $\mcL_\mcH$ and $\mathcal{C}_\mcH.$

For any $T \in \mcL$ the spectrum $\sigma(T)$ is contained in a bounded open region $\Omega = \Omega(T) \subset \IC$. Assume that $\Omega$ has a smooth boundary $\Gamma = \partial \Omega$ with
\[ \delta_{\Gamma, T} = \text{dist}(\Gamma, \sigma(T)) >0 .\]
Assume additionally that $\bar{\Omega} \subset D$ for an open set $D \subset \IC$ and that $\phi: D \to \IC$ is analytic. Define
\begin{equation}
\label{M,L_Gamma}
 M_\Gamma = \max_{z \in \Gamma}|\phi(z)| < \infty, \quad L_\Gamma = \text{ length of } \Gamma < \infty.
\end{equation}
On the resolvent set $\rho(T) = (\sigma(T))^c$, the resolvent given by 
\[  R(z,T) = (zI- T)^{-1}  \]
is well-defined and analytic. This allows to define the operator
\begin{equation}\label{e.tr4} \phi(T) = \frac{1}{2\pi i} \int_\Gamma \phi(z) R(z) \, \dx z .\end{equation}

Define additionally for $G \in \mcL$:
\begin{align}
 \label{eq:Dphi} D_T\phi[G] &= \frac{1}{2\pi i} \int_\Gamma \phi(z) R(z,T) G R(z,T) \, \dx z, \\
 \label{eq:D2phi} D^2_T\phi[G,G] &= \frac{1}{2\pi i} \int_\Gamma \phi(z) R(z,T) ( G R(z,T))^2 \, \dx z, \\
 \label{eq:S_phi} S_{\phi,T,2}[G] &= \frac{1}{2\pi i} \int_\Gamma \phi(z) R(z,T) ( G R(z,T))^3 (I-G R(z,T))^{-1} \, \dx z.
\end{align}
We will see in a moment that $D_T \phi$ and $D^2_T \phi$  are the first and second Fréchet derivatives of $\phi.$ The second derivative is a symmetric bilinear form. Recall that symmetric bilinear forms $B(\cdot, \cdot)$ are characterized by their corresponding quadratic form $Q(\cdot)$ via the polarization identity. 

By Lemma VII.6.11 in \cite{DS:88} there is a constant $K = |\Gamma| \sup_{z \in \Gamma} \| R(z,T) \| < \infty$ such that
\begin{equation}\label{eq:DS}
 \| R(z,T) \|_\mcL \leq \frac{K}{\delta_{\Gamma,T}}, \quad \forall z \in \Omega^c.
\end{equation}

Next, we derive an extension of Theorem 3.1 in \cite{GHJR:09}.
\begin{theorem}\label{thm:ghjr}
 Suppose that $\phi: D\subset \IC \to \R$ is analytic and $T \in \mcL$ with $\sigma(T) \subset \Omega(T)\subset D$ with 
 \[ \delta_{\Gamma, T} = \text{dist}(\Gamma, \sigma(T)) >0 .\]
 Then $\phi$ maps the neighborhood 
 \[ \{\tilde{T} = T+G:\, G \in \mcL, \|G\|_\mcL \leq c \delta_{\Gamma,T}/K \text{ for some } c<1\} \]
 into $\mcL$. This mapping is twice \Frechet{} differentiable at $T$, tangentially to $\mcL$, with bounded first derivative $D_T\phi: \mcL \to \mcL$ and the second derivative is characterized by its diagonal form $D^2_T\phi: \mcL \to \mcL$. More specifically, we have
 \begin{equation}\label{eq:phi:T+G}
  \phi(T+G) = \phi (T) + D_T\phi[G] + D^2_T\phi[G,G] + S_{\phi,T,2}[G]
 \end{equation}
 with
 \begin{equation}
  \| S_{\phi,T,2}[G] \|_\mcL  \leq \frac{1}{2(1-c)\pi}M_\Gamma L_\Gamma K^4 \delta_{\Gamma,T}^{-4} \|G\|_\mcL^3.
 \end{equation}
\end{theorem}

\begin{proof}
 We have  for all $G \in \mcL$ with $\|G\|_\mcL \leq c \delta_{\Gamma,T} K^{-1}$ by \eqref{eq:DS} that 
  \begin{equation}\label{eq:GRbound} 
  \| G R(z) \|_\mcL < c .
   \end{equation}
 This allows to calculate
 \begin{align}
  R(z,T) (I-GR(z,T))^{-1} &= R(z,T)\left[(R(z,T)^{-1} - G)R(z,T)\right]^{-1} \\
  \nn
  &=\left[ zI-T-G\right]^{-1} = R(z,\tilde T)
 \end{align}
 for any $z \in \Omega^c$ with $\tilde{T} = T + G$ as above. As the left hand side of the previous equation is well-defined, we conclude that 
 $z \in \rho(\tilde T)$. Thus, $\sigma(\tilde T) \subset \Omega$ and the mapping $\phi$ applied to $\tilde{T}=T+G$ is well defined via
 \begin{equation}\label{eq:phi:T+G:pr}
  \phi(T+G) = \frac{1}{2\pi i} \int_\Gamma \phi(z) R(z,T+G) \, \dx z \quad \text{ for } G \in \mcL \text{ with } \|G\|_\mcL \leq c\delta_{\Gamma,T}/K.
 \end{equation}
Using a Neumann series expansion we can obtain
\begin{align*}
 R(z,T+G) & = R(z,T)\big(I + GR(z,T)+(GR(z,T))^2 \\
          &\phantom{aaaaaaaaaa}  + (GR(z,T))^3(I-GR(z,T))^{-1} \big) \\
 & = R(z,T) + R(z,T)GR(z,T) + R(z,T)GR(z,T)GR(z,T) \\
 & \qquad + R(z,T) (GR(z,T))^3(I-GR(z,T))^{-1}.
\end{align*}
and inserting this into \eqref{eq:phi:T+G:pr} allows to obtain \eqref{eq:phi:T+G}. The bound on $S_{\phi,T,2}[G]$ can be obtained from 
(\ref{eq:S_phi}) using (\ref{M,L_Gamma}) as well as (\ref{eq:DS}) and $\|(I-GR(z,T))^{-1}\|_\mcL \leq (1-\|GR(z,T)\|_\mcL)^{-1}\leq (1-c)^{-1}$ by (\ref{eq:GRbound}).
\end{proof}

Now let us restrict $T$ to the subset $\mathcal{C}_\mcH$ of compact Hermitian operators. That allows a representation
\begin{equation}\label{eq:T:hermit}
 T = \sum_{i=1}^\infty \lambda_i P_i,
\end{equation}
where $\lambda_i \in \R$ are eigenvalues and $P_i$ are orthogonal projections onto one-dimensional eigenspaces (since $T$ is compact, to each non-zero eigenvalue there is a finite-dimensional eigenspace that can be decomposed into orthogonal spaces).
Then the resolvent has the following form
\begin{equation}\label{eq:reso:hermit}
 R(z,T) = \sum_{i=1}^\infty \frac{1}{z-\lambda_i} P_i, \quad z \in \rho(T)
\end{equation}
and for $\phi: D \subset \sigma(T) \to \R$:
\begin{equation}
 \phi(T) = \sum_{i=1}^\infty \phi(\lambda_i) P_i.
\end{equation}

\begin{corollary}
 \label{cor:Second derivative}
 Let the conditions of Theorem \ref{thm:ghjr} be fulfilled for $T \in \mathcal{C}_\mcH$ with expansion \eqref{eq:T:hermit}. In this case 
 \begin{align}
  &D_T\phi [G]   = \sum_{i=1}^\infty \phi'(\lambda_i) P_i G P_i + \sum_{i\neq k} \frac{\phi(\lambda_k)-\phi(\lambda_i)}{\lambda_k - \lambda_i} P_i G P_k
  \end{align}
  and
  \begin{align}
   \nn  D^2_T\phi & [G,G]  
   =    \sum_{i,j,k=1}^\infty \1_{\{\lambda_i \neq \lambda_j \neq \lambda_k \neq \lambda_i\}}
  P_i G P_j G P_k \\
\nn
 &  \phantom{aaaaaaaaaaaa} \cdot \frac{(\lambda_j-\lambda_k)\phi(\lambda_i) + (\lambda_k-\lambda_i)\phi(\lambda_j) + (\lambda_i-\lambda_j)\phi(\lambda_k)}{\lambda_i^2(\lambda_k-\lambda_j) + \lambda_j^2(\lambda_i-\lambda_k)  + \lambda_k^2(\lambda_j-\lambda_i) }\\
  & \quad  +  \sum_{i,j =1}^\infty  \frac{\1_{\{\lambda_j \neq \lambda_i\}}}{\lambda_i -\lambda_j}  \left( P_j G P_j G P_i + P_i G P_j G P_j + P_j G P_i G P_j \right) \\
  & \phantom{aaaaaaaaaaaa} \cdot \left[ \phi'(\lambda_j) - \frac{\phi(\lambda_i) -\phi(\lambda_j)}{\lambda_i -\lambda_j} \right]\\
  \nn &  \quad +\sum_{i=1}^\infty \phi''(\lambda_i) P_i G P_i G P_i .
 \end{align}
 for all $G \in \mcL.$
\end{corollary}
\begin{proof}
 We can use the explicit form of the resolvent from \eqref{eq:reso:hermit} in \eqref{eq:Dphi} and \eqref{eq:D2phi}. We restrict our attention to the second derivative since the first derivative was already explained in \cite{GHJR:09}. Thus,
 \begin{equation*}
  D^2_T\phi [G,G]  = \sum_{i,j,k=1}^\infty \frac{1}{2\pi i} \int_\Gamma \frac{\phi(z)}{(z-\lambda_i)(z-\lambda_j)(z-\lambda_k)} \, \dx z \ P_i G P_j G P_k.
 \end{equation*}
 Note that for pairwise different $\lambda_i, \lambda_j,  \lambda_k$:
 \begin{align*}
  \frac{1}{z-\lambda_i}\frac{1}{z-\lambda_j}\frac{1}{z-\lambda_k} & = \frac{1}{\lambda_i^2(\lambda_j-\lambda_k) + \lambda_j^2(\lambda_k-\lambda_i)  + \lambda_k^2(\lambda_i-\lambda_j) } \\
  & \qquad \qquad  \qquad   \quad \cdot \left[ \frac{\lambda_j-\lambda_k}{z-\lambda_i} + \frac{\lambda_k-\lambda_i}{z-\lambda_j} + \frac{\lambda_i-\lambda_j}{z-\lambda_k} \right].
 \end{align*}
Additionally, for $\lambda_i = \lambda_j \neq \lambda_k$:
 \begin{align*}
  \frac{1}{z-\lambda_i}\frac{1}{z-\lambda_j} &\frac{1}{z-\lambda_k} 
   = \frac{1}{\lambda_i -\lambda_k} \\
   &\ \cdot \left[ \frac{1}{(z-\lambda_i)^2} - \frac{1}{(z-\lambda_i)(\lambda_i-\lambda_k)} 
  + \frac{1}{(z-\lambda_k)(\lambda_i-\lambda_k)} \right].
 \end{align*}
 This allows to derive
 \begin{align*}
  D^2_T & \phi [G,G]  = \sum_{i,j,k = 1}^\infty \int_\Gamma \phi(z)  \frac{1}{z-\lambda_i}\frac{1}{z-\lambda_j}\frac{1}{z-\lambda_k} \, \dx z \ P_i G P_j G P_k \\
  & = \sum_{i,j,k=1}^\infty \1_{\{\lambda_i \neq \lambda_j \neq \lambda_k \neq \lambda_i\}} P_i G P_j G P_k \\
  & \phantom{AAAAAAAAAA}  \frac{(\lambda_j-\lambda_k)\phi(\lambda_i) + (\lambda_k-\lambda_i)\phi(\lambda_j) + (\lambda_i-\lambda_j)\phi(\lambda_k)}
  {\lambda_i^2(\lambda_j-\lambda_k) + \lambda_j^2(\lambda_k-\lambda_i)  + \lambda_k^2(\lambda_i-\lambda_j) }   \\
  & \quad +\sum_{i,k =1}^\infty \1_{\{\lambda_i = \lambda_j \neq \lambda_k\}} \frac{1}{\lambda_i -\lambda_k} 
  \left[ \phi'(\lambda_i) - \frac{\phi(\lambda_i) -\phi(\lambda_k)}{\lambda_i -\lambda_k} \right] P_i G P_i G P_k \\
  & \quad +\sum_{i,j =1}^\infty \1_{\{\lambda_j = \lambda_k \neq \lambda_i\}} \frac{1}{\lambda_j -\lambda_i} 
  \left[ \phi'(\lambda_j) - \frac{\phi(\lambda_j) -\phi(\lambda_i)}{\lambda_j -\lambda_i} \right] P_i G P_j G P_j 
  \\ 
  & \quad +\sum_{j,k =1}^\infty \1_{\{\lambda_k = \lambda_i \neq \lambda_j\}} \frac{1}{\lambda_k -\lambda_j} 
  \left[ \phi'(\lambda_k) - \frac{\phi(\lambda_k) -\phi(\lambda_j)}{\lambda_k -\lambda_j} \right] 
   P_k G P_j G P_k \\
  & \quad +\sum_{i=1}^\infty  \phi''(\lambda_i) P_i G P_i G P_i \, .
  \end{align*}
  Now a relabeling of the indices allows to obtain the result we wanted to show.
\end{proof}

\section{Some elementary facts on matrices}

The next results are elementary but as we regularly use them we state them here.
\begin{lemma}[Theorem 2.8 of \cite{fZ:11}]\label{lem:ev(AB)}
 Let $A$ and $B$ be $m\times n$ and $n\times m$ complex matrices, respectively. Then $AB$ and $BA$ have the same non-zero eigenvalues, counting multiplicity. In particular for symmetric positive definite $\Sigma$ and $\Xi$: eigenvalues of $(A^{1/2}BA^{1/2})$ and $(AB)$ are the same, counting multiplicity. Moreover,
 \begin{equation} \tr (AB) =\tr(BA) . \end{equation}
\end{lemma}
%

A helpful tool for calculating the trace is the following lemma.
\begin{lemma}\label{lem:trace-formula}
 Let $\{x_1, \dots, x_\dimension \}$ be any orthonormal basis of $\Rdim$ and $A \in \R^{\dimension \times \dimension}$. Then
 \begin{equation} \tr (A) = \sum_{i=1}^\dimension x_i^t A x_i. \end{equation}
\end{lemma}
\begin{proof}
 Let $P=(x_1^t, \dots, x_\dimension^t )^t\in \R^{\dimension \times \dimension},$ so the first row of $P$ is $x_1$ and so on. Then $P^tP = 1,$ i.e.~$P$ is unitary and thus,
 \begin{align*}
  \tr (A) &= \sum_{j=1}^\dimension A_{jj} = \sum_{j,k=1}^\dimension \delta_{k=j} A_{jk}  = \sum_{j,k=1}^\dimension (P^tP)_{kj} A_{jk}  \\
  & = \sum_{i,j,k=1}^\dimension P^t_{ki} A_{kj}P_{ij} = \sum_{i=1}^\dimension x_i^t A x_i.
 \end{align*}
\end{proof}

Recall the matrix $H$ of Lemma \ref{lem:anderson}.
It is the prototype of matrix which appears in the next lemma.%
\begin{lemma}\label{lem:EH}
Let $H \in \R^{d\times d}$ be symmetric with independent centered Gaussian entries in the upper triangular part s.t.~$H_{ii} \sim N(0,2)$ for $1\leq i \leq d$ and $H_{ij} \sim N(0,1)$ for $1\leq i < j \leq d$. Let $m,n \in \IN$. For $C \in \R^{m\times d}$, $D \in \R^{d \times d}$ and $E\in \R^{d \times n}$ it holds that
 \begin{equation}
 \IE[ \left( CHDHE \right)_{ij}] = \left(CD^tE \right)_{ij} + \left(CE \right)_{ij} \cdot \tr (D) \, , \quad 1\leq i \leq m, \, 1\leq j \leq n.
 \end{equation}
\end{lemma}
\begin{proof}
We note that $1\leq k,l,p,q \leq d$ we have
\begin{equation*} \IE[H_{kl}H_{pq}] = 2\1_{\{k=l=p=q\}} + \1_{\{k=p\neq l=q\}} + \1_{\{k=q \neq l=p\}} \, . 
\end{equation*}
We can use that on the matrix product
\begin{equation*}
   \left( CHDHE \right)_{ij} = \sum_{k,l,p,q = 1}^d C_{ik}H_{kl}D_{lp}H_{pq}E_{qj}
 \end{equation*}
to evaluate
\begin{align*}
   \IE[&\left(CHDHE \right)_{ij}] = 2 \sum_{k,l,p,q = 1}^d \1_{\{k=l=p=q\}} C_{ik}D_{lp}E_{qj} \\
  & \quad + \sum_{k,l,p,q = 1}^d \1_{\{k=p\neq l=q\}} C_{ik}D_{lp}E_{qj} \\
  & \quad + \sum_{k,l,p,q = 1}^d \1_{\{k=q \neq l=p\}} C_{ik}D_{lp}E_{qj} \\
  & = 2 \sum_{k=1}^d C_{ik}D_{kk}E_{kj} + \sum_{k=1}^d \sum_{l=1,l\neq k}^d C_{ik}D_{lk}E_{lj} + \sum_{k=1}^d\sum_{l=1,l\neq k}^d C_{ik}D_{ll}E_{kj} \\
  & = \left(CD^tE \right)_{ij} + \left(CE \right)_{ij} \cdot \tr (D) \, .
\end{align*}
\end{proof}

\bibliographystyle{plain}
\bibliography{propos_1}

\begin{thebibliography}{10}

\bibitem{Loubes}
Marina Agull\'o-Antol\'in, J.A. Cuesta-Albertos, H\'el\`{e}ne Lescornel, and
  Jean-Michel Loubes.
\newblock A parametric registration model for warped distributions with
  {W}asserstein's distance.
\newblock {\em J.~Multivar. Anal.}, 135(0):117--130, 2015.

\bibitem{AKT:84}
Mikl\'os Ajtai, J\'anos Koml{\'o}s, and G\'abor Tusn{\'a}dy.
\newblock On optimal matchings.
\newblock {\em Combinatorica}, 4(4):259--264, 1984.

\bibitem{ABCM08}
Pedro~Cesar Alvarez-Esteban, Eustasio {d}el Barrio, Juan~Alberto
  Cuesta-Albertos, and Carlos Matr\'an.
\newblock Trimmed comparison of distributions.
\newblock {\em J. American Stat. Ass.}, 103(482):697--704, 2008.

\bibitem{AS:67}
V.~I. Averbukh and O.~G. Smolyanov.
\newblock The theory of differentiation in linear topological spaces.
\newblock {\em Russian Mathematical Surveys}, 22(6):201--258, 1967.

\bibitem{Berman01012000}
Helen~M. Berman, John Westbrook, Zukang Feng, Gary Gilliland, T.~N. Bhat, Helge
  Weissig, Ilya~N. Shindyalov, and Philip~E. Bourne.
\newblock The protein data bank.
\newblock {\em Nucleic Acids Research}, 28(1):235--242, 2000.

\bibitem{bickel1997resampling}
P.J. Bickel, F.~G\"otze, and W.R. van Zwet.
\newblock Resampling fewer than n observations: Gains, losses, and remedies for
  losses.
\newblock In Sara van~de Geer and Marten Wegkamp, editors, {\em Selected Works
  of Willem van Zwet}, Selected Works in Probability and Statistics, pages
  267--297. Springer New York, 2012.

\bibitem{boissard2011distribution}
Emmanuel Boissard, Thibaut Le~Gouic, and Jean-Michel Loubes.
\newblock Distribution's template estimate with {W}asserstein metrics.
\newblock {\em Bernoulli}, 21(2):740--759, 2015.

\bibitem{Cheney:01}
Ward Cheney.
\newblock {\em Analysis for applied mathematics}, volume 208 of {\em Graduate
  Texts in Mathematics}.
\newblock Springer-Verlag, New York, 2001.

\bibitem{Cuesta1996lower}
J.A. Cuesta-Albertos, C.~Matr\'an-Bea, and A.~Tuero-Diaz.
\newblock On lower bounds for the {L}2-{W}asserstein metric in a {H}ilbert
  space.
\newblock {\em J. Theoret. Probab.}, 9(2):263--283, 1996.

\bibitem{czado_assessing_1998}
Claudia Czado and Axel Munk.
\newblock Assessing the similarity of distributions-finite sample performance
  of the empirical {Mallows} distance.
\newblock {\em Journal of Statistical Computation and Simulation},
  60(4):319--346, 1998.

\bibitem{davison1997bootstrap}
Anthony~Christopher Davison.
\newblock {\em Bootstrap methods and their application}, volume~1.
\newblock Cambridge {U}niversity {P}ress, 1997.

\bibitem{del2000contributions}
Eustasio del Barrio, Juan~A Cuesta-Albertos, Carlos Matr{\'a}n, S{\'a}ndor
  Cs{\"o}rg{\"o}, Carles~M Cuadras, Tertius de~Wet, Evarist Gin{\'e}, Richard
  Lockhart, Axel Munk, and Winfried Stute.
\newblock Contributions of empirical and quantile processes to the asymptotic
  theory of goodness-of-fit tests.
\newblock {\em Test}, 9(1):1--96, 2000.

\bibitem{del_barrio_central_1999}
Eustasio {d}el Barrio, Evarist Gin\'e, and Carlos Matr\'an.
\newblock Central limit theorems for the {{W}asserstein} distance between the
  empirical and the true distributions.
\newblock {\em Ann. Probab.}, 27(2):1009--1071, 1999.

\bibitem{del2005asymptotics}
Eustasio {d}el Barrio, Evarist Gin{\'e}, Frederic Utzet, et~al.
\newblock Asymptotics for {L}2 functionals of the empirical quantile process,
  with applications to tests of fit based on weighted {W}asserstein distances.
\newblock {\em Bernoulli}, 11(1):131--189, 2005.

\bibitem{dBM13-2}
Eustasio del Barrio and Carlos Matr\'an.
\newblock Rates of convergence for partial mass problems.
\newblock {\em Probab. Theory Related Fields}, 155:521--542, 2013.

\bibitem{Demarta}
Stefano Demarta and Alexander~J. McNeil.
\newblock The t copula and related copulas.
\newblock {\em International Statistical Review}, 73(1):111--129, 2005.

\bibitem{D:69}
Jean Dieudonn{\'e}.
\newblock {\em Foundations of modern analysis}.
\newblock Academic Press, New York-London, 1969.

\bibitem{dobric_asymptotics_1995}
Vladimir~T. Dobri\'c and Joseph~E. Yukich.
\newblock Asymptotics for transportation cost in high dimensions.
\newblock {\em J. Theoret. Probability}, 8(1):97--118, 1995.

\bibitem{dowson1982frechet}
D.C. Dowson and B.V. Landau.
\newblock The {F}r{\'e}chet distance between multivariate normal distributions.
\newblock {\em J. {M}ultivar. {A}nal.}, 12(3):450--455, 1982.

\bibitem{DS:88}
Nelson Dunford and Jacob~T. Schwartz.
\newblock {\em Linear operators. {P}art {I}}.
\newblock Wiley Classics Library. John Wiley \& Sons Inc., New York, 1988.

\bibitem{fournier_rate_2013}
Nicolas Fournier and Arnaud Guillin.
\newblock On the rate of convergence in {{W}asserstein} distance of the
  empirical measure.
\newblock {\em Probab. Theory Related Fields}, pages 1--32.{ }, 2014.

\bibitem{Frechet1951}
Maurice Fr{\'e}chet.
\newblock Sur les tableaux de corr{\'e}lation dont les marges son donn{\'e}es.
\newblock {\em Ann. Univ. Lyon, Sect. A}, 9:53--77, 1951.

\bibitem{freitag2007nonparametric}
Gudrun Freitag, Claudia Czado, and Axel Munk.
\newblock A nonparametric test for similarity of marginals with applications to
  the assessment of population bioequivalence.
\newblock {\em J. Stat. Plan. Inference}, 137(3):697--711, 2007.

\bibitem{freitag2005hadamard}
Gudrun Freitag and Axel Munk.
\newblock On {H}adamard differentiability in k-sample semiparametric models
  with applications to the assessment of structural relationships.
\newblock {\em J. {M}ultivar. {A}nal.}, 94(1):123--158, 2005.

\bibitem{gelbrich1990formula}
Matthias Gelbrich.
\newblock On a formula for the {L}2 {W}asserstein metric between measures on
  {E}uclidean and {H}ilbert spaces.
\newblock {\em Mathematische Nachrichten}, 147(1):185--203, 1990.

\bibitem{GHJR:09}
David~S. Gilliam, Thorsten Hohage, Xiaoyi Ji, and Frits~H. Ruymgaart.
\newblock The {F}r\'echet derivative of an analytic function of a bounded
  operator with some applications.
\newblock {\em Int. J. Math. Math. Sci.}, pages Art. ID 239025, 17 pages, 2009.

\bibitem{GS:84}
Clark~R. Givens and Rae~Michael Shortt.
\newblock A class of {W}asserstein metrics for probability distributions.
\newblock {\em Michigan Math. J.}, 31(2):231--240, 1984.

\bibitem{Hoeffding1940}
Wassily Hoeffding.
\newblock Ma{\ss}stabinvariante {K}orrelationstheorie.
\newblock {\em Schriften des Mathematischen Instituts und des Instituts f\"ur
  Angewande Mathematik der Universit{\"a}t Berlin}, 5:179--233, 1940.

\bibitem{Honndorf:12}
V.~S. Honndorf, N.~Coudevylle, S.~Laufer, S.~Becker, C.~Griesinger, and
  M.~Habeck.
\newblock Inferential {NMR/X}-ray based structure determination of a
  dibenzo[a,d]cyclo-heptenone inhibitor/p38 {MAP} kinase complex in solution.
\newblock {\em Angewandte Chemie}, 51:2359--2362, 2012.

\bibitem{KR58}
L.~V. Kantorovich and G.~S. Rubinshtein.
\newblock On a space of totally additive functions.
\newblock {\em Vestn. Leningrad. Univ.}, 13(7):52--59, 1958.

\bibitem{knott1984optimal}
M.~Knott and C.S. Smith.
\newblock On the optimal mapping of distributions.
\newblock {\em J. Optimiz. Theory. App.}, 43(1):39--49, 1984.

\bibitem{Major1978487}
P\'eter Major.
\newblock On the invariance principle for sums of independent identically
  distributed random variables.
\newblock {\em J. {M}ultivar. {A}nal.}, 8(4):487 -- 517, 1978.

\bibitem{Mallows}
C.L. Mallows.
\newblock A note on asymptotic joint normality.
\newblock {\em Ann. Math. Statist.}, 43:508--515, 1972.

\bibitem{munk_nonparametric_1998}
Axel Munk and Claudia Czado.
\newblock Nonparametric validation of similar distributions and assessment of
  goodness of fit.
\newblock {\em J. R. Stat. Soc. (B)}, 60(1):223--241, 1998.

\bibitem{CM:98}
Axel Munk and Claudia Czado.
\newblock Nonparametric validation of similar distributions and assessment of
  goodness of fit.
\newblock {\em J. R. Stat. Soc. B}, 60(1):223--241, 1998.

\bibitem{olkin1982distance}
Ingram Olkin and Friedrich Pukelsheim.
\newblock The distance between two random vectors with given dispersion
  matrices.
\newblock {\em Linear Algebra Appl.}, 48:257--263, 1982.

\bibitem{rachev1998mass}
Svetlozar~T Rachev and Ludger R{\"u}schendorf.
\newblock {\em Mass Transportation Problems: Volume I: Theory}, volume~1.
\newblock Springer-Verlag, 1998.

\bibitem{ruschendorf1990characterization}
Ludger R{\"u}schendorf and Svetlozar~T Rachev.
\newblock A characterization of random variables with minimum {L}2-distance.
\newblock {\em J. {M}ultivar. {A}nal.}, 32(1):48--54, 1990.

\bibitem{ruttenberg2013quantifying}
Brian~E Ruttenberg, Gabriel Luna, Geoffrey~P Lewis, Steven~K Fisher, and
  Ambuj~K Singh.
\newblock Quantifying spatial relationships from whole retinal images.
\newblock {\em Bioinformatics}, 29(7):940--946, 2013.

\bibitem{RY:97}
Frits~H. Ruymgaart and Song Yang.
\newblock Some applications of {W}atson's perturbation approach to random
  matrices.
\newblock {\em J. Multivar. Anal.}, 60(1):48--60, 1997.

\bibitem{shao1995jackknife}
Jun Shao and Dongsheng Tu.
\newblock {\em The jackknife and bootstrap}.
\newblock Springer, 1995.

\bibitem{talagrand_matching_1992}
Michel Talagrand.
\newblock Matching random samples in many dimensions.
\newblock {\em Ann. Appl. Probab.}, pages 846--856, 1992.

\bibitem{Trueblood:es0238}
K.~N. Trueblood, H.-B. B{\"{u}}rgi, H.~Burzlaff, J.~D. Dunitz, C.~M.
  Gramaccioli, H.~H. Schulz, U.~Shmueli, and S.~C. Abrahams.
\newblock {Atomic Dispacement Parameter Nomenclature. Report of a Subcommittee
  on Atomic Displacement Parameter Nomenclature}.
\newblock {\em Acta Crystallographica Section A}, 52(5):770--781, Sep 1996.

\bibitem{Vallender}
S.~S. Vallender.
\newblock Calculation of the {W}asserstein distance between probability
  distributions on the line.
\newblock {\em Theory Probab. Appl.}, 18(4):784--786, 1974.

\bibitem{vdV:98}
Aad~W. van~der Vaart.
\newblock {\em Asymptotic statistics}, volume~3 of {\em Cambridge Series in
  Statistical and Probabilistic Mathematics}.
\newblock Cambridge University Press, Cambridge, 1998.

\bibitem{vdVW:96}
Aad~W. van~der Vaart and Jon~A. Wellner.
\newblock {\em Weak convergence and empirical processes}.
\newblock Springer Series in Statistics. Springer-Verlag, New York, 1996.

\bibitem{cV:07}
C{\'e}dric Villani.
\newblock {\em Optimal transport}, volume 338 of {\em Grundlehren der
  Mathematischen Wissenschaften}.
\newblock Springer-Verlag, Berlin, 2009.

\bibitem{fZ:11}
Fuzhen Zhang.
\newblock {\em Matrix theory}.
\newblock Universitext. Springer, New York, second edition, 2011.

\bibitem{zhou2011statistical}
Dunke Zhou and Tao Shi.
\newblock Statistical inference based on distances between empirical
  distributions with applications to airslevel-3 data.
\newblock In {\em CIDU}, pages 129--143, 2011.

\end{thebibliography}

\end{document}